\documentclass[11pt]{amsart}
\usepackage{fullpage}
\allowdisplaybreaks 
\usepackage{todonotes}
\usepackage{hyperref}
\usepackage{amsmath,amsfonts,amssymb}
\usepackage{mathrsfs}
\usepackage{verbatim}
\usepackage{amsthm}
\usepackage{mathrsfs} 
\usepackage{import}

\newtheorem{theorem}{Theorem}[section]
 \newtheorem{corollary}[theorem]{Corollary}
 \newtheorem{lemma}[theorem]{Lemma}
 \newtheorem{proposition}[theorem]{Proposition}
 \theoremstyle{definition}

 \theoremstyle{remark}
 \newtheorem{remark}[theorem]{Remark}
  \newtheorem{ex}[theorem]{Example}
 \numberwithin{equation}{section}
 
\def \bC {\mathbb C}

\def \bH {\mathbb H}

\def \bN {\mathbb N}

\def \bR {\mathbb R}
\def \bS {\mathbb S}
\def \bT {\mathbb T}

\def \bZ {\mathbb Z}
\def \Z {\mathbb Z}

\def \R {\mathbb R}

\def \cD {\mathcal D}

\def \cF {\mathcal F}

\def \cH {\mathcal H}

\def \cL {\mathcal L}

\def \cS {\mathcal S}

\def \cU {\mathcal U}

\def \fg {\mathfrak g}
\def \fh {\mathfrak h}

\def \sL{\mathscr L}

\def \ad {{\rm ad}}
\def \id {\text{\rm I}}

\def \Gh {{\widehat G}}
\def \eps {\varepsilon}
\def \vol {{\rm vol}}

\def \spec  {{\rm sp}\, }

\def \BWZ {{\rm BWZ}}

\def \Pr {{\rm{Pr}}}

\newcommand{\ds}{\displaystyle}

\title
{Restriction problems\\ on the three-dimensional Heisenberg nilmanifold}    
\date{\today}
\author[H.  Bahouri  \&  V. Fischer]
{Hajer Bahouri and V\'eronique Fischer}

\address[H. Bahouri] 
{CNRS  \&  Sorbonne Universit\'e  \\
 Laboratoire Jacques-Louis Lions (LJLL) UMR  7598 \\
4, Place Jussieu\\
75005 Paris, France.}
\email{hajer.bahouri@sorbonne-universite.fr}
 
\address[V. Fischer]
{University of Bath, Department of Mathematical Sciences, Bath, BA2 7AY, UK} 
\email{v.c.m.fischer@bath.ac.uk}



\begin{document}

\begin{abstract}
In this paper, we prove a spectral restriction theorem on the three-dimensional
Heisenberg nilmanifold. Since this manifold is an $\mathbb S^1$-bundle over the flat torus
$\mathbb T^2$, the result provides a sub-elliptic counterpart of Zygmund's
restriction theorem on  $\mathbb T^2$ \cite{zygmund}. 
We also establish its
sharpness by means of the discrete short-time Fourier transform.
\end{abstract}

\maketitle

\setcounter{tocdepth}{2}

\makeatletter
\renewcommand\l@subsection{\@tocline{2}{0pt}{3pc}{5pc}{}}
\makeatother
 
\tableofcontents

\section{Introduction}

\subsection{Restriction problems on the torus}
We start
with the classical inequality on the two-dimensional torus $\bT^2$:
 \begin{equation}   \label{zyg2ddual}
\Big\|\sum_{|\omega|=r}c_\omega \chi_\omega \Big\|_{L^{4}(\bT^2)} \leq 5^{\frac 14} \|(c_\omega)_{\omega\in \bZ^2}\|_{\ell^2(\bZ^2)},\end{equation} 
valid for any $r>0$ and $(c_\omega)_{\omega\in \bZ^2} \in \ell^2(\bZ^2)$, and 
where we have denoted for any $\omega \in \Z^2$,   
$$
|\omega|:=\sqrt{\omega_1^2+\omega_2^2}
\qquad\mbox{and}\qquad 
\chi_\omega (x,y):= e^{ 2\pi i (\omega_1 x+\omega_2 y)}, \ (x,y)\in \bT^2.
$$
This estimate, or rather its dual, was proved by Zygmund \cite{zygmund} in the 1970's.
Zygmund's result may be summarised as 
\begin{equation}
\label{eq_zygsumm}	
\|\Pi_{r,\sqrt{\Delta}}\|_{\sL(L^{2}(\bT^2), L^4(\bT^2))} 
=\|\Pi_{r,\sqrt{\Delta}}\|_{\sL(L^{4/3}(\bT^2), L^2(\bT^2))} 
\leq 5^{1/4}.
\end{equation}
Above, the operator $\Pi_{r,\sqrt{\Delta}}$ is given by 
\begin{equation}   \label{proj2d}
\Pi_{r,\sqrt{\Delta}} f = \sum_{|\omega|=r} (f,\chi_\omega)_{L^2(\bT^2)} \chi_\omega, 
\qquad  f\in L^2(\bT^2).
\end{equation}  
This is the spectral projector onto the $r$-eigenspace for $\sqrt{\Delta}$ where $\Delta=-\partial_1^2-\partial_2^2$ 
denotes the standard Laplacian  
on the two-dimensional torus~$\bT^2$.

\smallskip This result is specific to the two-dimensional case: as we will see in Section~\ref{sketch}, its  proof  relies  mainly on an underlying simple arithmetic structure.  This is no longer the case in higher dimensions:  for $\Pi_{r,\sqrt{\Delta}}$, with $\Delta=-\partial_1^2-\cdots-\partial_n^2$ on the~$n$-dimensional torus~$\bT^n$, it has been  established by Bourgain~\cite{bourgain} and Bourgain-Demeter~\cite{bourgain-demeter1} that, for all $n \geq 4$, $1 \leq p \leq \frac {2n}{n+3}$ and $\epsilon>0$, there holds
\begin{equation}   \label{zygnd}
\|\Pi_{r,\sqrt{\Delta}}\|_{\sL(L^{p}(\bT^n), L^2(\bT^n))} \lesssim_\epsilon r^{n  (\frac 1 p- \frac 1 2) - 1 + \epsilon}\, ,
\end{equation}  
while for $n=3$,  $1 \leq p \leq \frac {3}{2}$ and $\epsilon>0$, Bourgain proved in \cite{bourgain1} the following  estimate
\begin{equation}   \label{zygn3}
\|\Pi_{r,\sqrt{\Delta}}\|_{\sL(L^{p}(\bT^3), L^2(\bT^3))} \lesssim_\epsilon r^{\epsilon}\,.
\end{equation} 

 Such   problems are known as discrete restriction problems.
Given the wide range of their applications,  they have been extensively studied in recent decades and we  discuss a brief historical overview  in Section \ref{brief}  below. 
For the latest  developments on this topic, we refer the interested reader to    the    surveys  by Guth \cite{Guth} and  Demeter \cite{DemeterICM, DemeterBourgain}.


\subsection{Setting of the problem and statement of the main results}

The aim of this paper is to show a discrete restriction problem for the three-dimensional Heisenberg nilmanifold.  
Here, the 3-dimensional Heisenberg group $\bH_1$ is realised as the set of  $3\times 3$-matrices with real entries, upper triangular and unipotent (that is,  with entries equal to 1 on the diagonal):
\begin{equation}\label{rea-H1}
\left(
\begin{array}{ccc}
	1&     a&c\\0&1&b\\0&0&1
\end{array}	\right) , \qquad a,b,c\in \bR. 
\end{equation}
We consider its subgroup $\Gamma$ consisting of the elements above corresponding to integer entries~$a,b,c\in \bZ$. 
The standard three-dimensional Heisenberg nilmanifold is the resulting right quotient:
$$
M:=\Gamma \backslash \bH_1.
$$
It is a compact smooth manifold of dimension $3$ that may be described as an $\bS^1$-bundle over the flat torus~$\bT^2$.
 A fundamental fact  is that 
 $M$ is equipped   with a measure,  a sub-Riemannian structure, and a  canonical sub-Laplacian  $\cL_M$;
this  will be described   in Section \ref{def-prop}. The spectral decomposition of  $\cL_M$  is well-known from a spectral viewpoint  \cite{DeningerSinghof} and in  harmonic analysis~\cite{KTX,Th2009}; it will be recalled in Section \ref{spectral}.

\medskip 
In this paper, we will prove the following result:
\begin{theorem}\label{main}
Considering the canonical sub-Laplacian $\cL_M$ on $M$, 
we denote by $\Pi_{\mu,\sqrt{\cL_M}}$ the spectral projector of $\sqrt{\cL_M}$. We have for all $\mu>0$ and $\epsilon>0$
\begin{equation}   \label{Nilrest}
\|\Pi_{\mu,\sqrt{\cL_M}}\|_{\sL(L^2(M),L^4(M))}
=\|\Pi_{\mu,\sqrt{\cL_M}}\|_{\sL(L^{4/3}(M),L^2(M))}
 \leq C_\epsilon \mu^{\frac 1 2  +\epsilon} .
\end{equation}    \end{theorem}

\smallskip  

To our knowledge, 
Theorem \ref{main} is the first restriction result in the challenging context of nilmanifolds, which is both spectrally discrete and sub-elliptic. 
The estimate \eqref{Nilrest} is sharp up to the factor $\mu^{\epsilon}$. More precisely, we will also show the following proposition: 
\begin{proposition} \label{prop:sharpness}
The exponent $1/2$ in Theorem \ref{main} is optimal in the sense that there exists $C>0$ such that we have for all $\mu>0$ in the spectrum of $\sqrt{\cL_M}$
\[
\|\Pi_{\mu,\sqrt{\cL_M}}\|_{\sL(L^2(M),L^4(M))}
=\|\Pi_{\mu,\sqrt{\cL_M}}\|_{\sL(L^{4/3}(M),L^2(M))}
 \geq C \mu^{\frac 1 2 } .
 \]  
\end{proposition}  

Compared to Zygmund's result \eqref{eq_zygsumm}, 
the factor $\mu^{\frac 12}$ can be interpreted as a fractional loss of derivatives.
Here, the notion  of derivative is associated to the nilmanifold~$M$ and its canonical sub-Riemannian structure; it can also be defined in terms of~$\sqrt{\cL_M}$ and its functional analysis. 
 This concept goes back to the   seventies with the founding papers of H\"ormander~\cite{hormander4} and Rothschild-Stein  \cite{RLS}. 
 However, it is not clear whether the  factor~$\mu^{\epsilon}$ in \eqref{Nilrest} is unavoidable: it  occurs often in discrete frameworks (see e.g. \eqref{zygnd} and \eqref{zygn3}), and  arises with questions related to number theory. This is also the case in our setting.

\smallskip

Theorem \ref{main}
 improves significantly the trivial estimate  given by the so-called Bernstein inequality in this context:
 \begin{proposition} [Sub-elliptic Bernstein Inequality]
 \label{proptrivial} 
With the  notation of Theorem \ref{main}, there exists a  constant $C>0$ such that the following estimate holds for any $\mu>0$:
 \begin{equation}   \label{Niltrivialrest}
 \|\Pi_{\mu,\sqrt{\cL_M}}\|_{\sL(L^2(M),L^4(M))}
=\|\Pi_{\mu,\sqrt{\cL_M}}\|_{\sL(L^{4/3}(M),L^2(M))}
\leq C  \mu.
\end{equation}  

\end{proposition}

The proof of the sub-elliptic Bernstein inequality  will be given in greater generality in Theorem \ref{thm_Bernstein}, see also Remark \ref{rem_pfproptrivial} for the case of the Heisenberg nilmanifold. 
 It should be emphasised that the sub-elliptic Bernstein inequality is governed not by the topological dimension of $M$, but by the homogeneous dimension of $\bH_1$, which is equal to four.

 \subsection{A brief historical overview on restriction problems}  \label{brief}   
 \subsubsection{The Euclidean framework}  
Fourier restriction's problem which was introduced by E. Stein dates back to the 60s:   given a hypersurface~$\hat S \subset \hat\R^n$ endowed with a smooth measure~$d\sigma$, the restriction problem  asks for which~$p$ an inequality of the form
\begin{equation}\label{eq:estime0}
 \|\cF( f)|_{\hat S}\|_{L^2 (\hat S, d \sigma)}\leq C \|f\|_{L^p (\R^n)}\, ,
\end{equation}
holds for all $f$ in  the Schwartz space $\cS(\R^n)$, where  $\cF (f)$ denotes the Fourier transform of  $f$ on~$\R^n$. When the Gaussian curvature of $\hat S$ does not vanish at every point, this problem is fully understood, see \cite{Fefferman1, stein, strichartz, Tomas} and  the optimal range of indices $p$ satisfying Inequality~\eqref{eq:estime0} is $[1, p_{TS}]$  with $p_{TS}$ the Tomas-Stein index given by $p_{TS}=(2n+2)/(n+3)$. It turns out that this problem is closely related to many questions in  harmonic analysis,  partial differential equations,  Fourier analysis, spectral theory, number theory, and many others. As a matter of fact, this type of problem, which  has been  generalised to other hyper-surfaces in Euclidean spaces and other settings,    remains a topical issue to this day. We refer the interested reader  to the  surveys on the subject   with various  points of view   \cite{HF, DemeterBourgain, Guth, Tao} and  the references therein.

\smallskip 
In this work, we focus on the case of the nilmanifold $M$ for which a Fourier transform is unavailable. We will then adopt the reformulation of  the problem of Fourier restriction  into a spectral problem as featured for instance in \cite{HF, Sogge2017}. As we shall see in Section \ref{spectral}, the spectral decomposition of the  sub-Laplacian $\cL_M$   includes the spectral decomposition of the Laplacian on    the two-dimensional flat torus given by \eqref{proj2d}. So, we dedicate the following  paragraph to the strategy of proof of Zygmund's theorem \cite{zygmund} regarding the   Fourier restriction's problem on the two dimensional torus. This emphasises  the link between the restriction problem in the discrete case and number theory.

  \subsubsection{Sketch of the proof of Zygmund's theorem}  \label{sketch} 
  
  Here we present the ideas of Zygmund's proof for  \eqref{zyg2ddual},
 and choose to give the details of the arithmetic  considerations sketched in~\cite{zygmund}.
We start with 
\begin{align*}
	\Big\|\sum_{2\pi |\omega|=r}c_\omega \chi_\omega \Big\|_{L^{4}(\bT^2)}^4 
	&= \Big\|\sum_{ |\omega_1| = |\omega_2|=\frac{r}{2\pi} } c_{\omega_1}\overline{c_{\omega_2}}\chi_{\omega_1}\overline{\chi_{\omega_2}} \Big\|_{L^2(\bT^2)}^2
	= \Big\|\sum_{ |\omega_1|= |\omega_2|=\frac{r}{2\pi}}  c_{\omega_1}\overline{c_{\omega_2}}\chi_{\omega_1-\omega_2}\Big\|_{L^2(\bT^2)}^2
	\\ 
	&=\|(\Gamma_{\rho,r})_{\rho\in \bZ^2} \|^2_{\ell^{2}(\bZ^2)}  \, , 
	 \qquad\mbox{where}\qquad  
   \Gamma_{\rho,r}:= \sum_{\substack{\omega_1-\omega_2=\rho\\ |\omega_1|=\frac{r}{2\pi} = |\omega_2|}} c_{\omega_1}\overline{c_{\omega_2}}\, ,
\end{align*}
by the Plancherel formula.
We have for $\rho=0$
by definition, 
$$
\Gamma_{0,r}=\sum_{|\omega|=\frac r {2\pi}} |c_\omega|^2 =|\Gamma_{0,r}| \leq  \sum_{\mu\in \bZ^2}
|c_\mu|^2 =\|(c_\omega)\|^2_{\ell^{2}(\bZ^2)},
$$ 
while 
for any $\rho \neq 0  $  
by the Cauchy-Schwartz inequality, 
$$
|\Gamma_{\rho,r}|^2\leq 
\sum_{\substack{\omega_1-\omega_2=\rho\\ |\omega_1|=\frac{r}{2\pi} = |\omega_2|}} |c_{\omega_1}\overline{c_{\omega_2}}|^2
\sum_{\substack{\omega_1-\omega_2=\rho\\ |\omega_1|=\frac{r}{2\pi} = |\omega_2|}} 1.
$$
We observe that  when $\rho\in \bR^2\setminus\{0\}$ and $r'\geq \frac 1 2$ are given, 
we can compute explicitly the solutions  $\omega_1,\omega_2\in \bR^2$ for
\begin{align}
\label{eq:Z} &|\omega_1|= |\omega_2| = r',
\ \omega_1-\omega_2=\rho 
\\& \nonumber \qquad \Longleftrightarrow 
\omega_1 = \frac \rho2  +\eps_0\sqrt{{r'}^2-\frac 14}\rho^\perp ,\ \omega_2 = -\frac \rho2  +\eps_0\sqrt{{r'}^2-\frac 14}\rho^\perp, \ \eps_0=\pm 1,
\end{align}
where $\rho^\perp := (-\rho_2,\rho_1)$ when we write $\rho = (\rho_1,\rho_2)\in \bR^2$, while  when $\rho\in \bR^2\setminus\{0\}$ and~$0 \leq r'< \frac 1 2$, the set of solutions of \eqref{eq:Z}  is empty. 
Consequently, summing over $\rho\in \bZ\setminus\{0\}$, 
$$
\sum_{\rho\neq 0} 
|\Gamma_{\rho,r}|^2
\leq 2 
\sum_{\omega,\rho\in \bZ^2}
|c_\omega|^2 
\left(
|c_{\omega-\rho}|^2+|c_{\omega+\rho}|^2\right)
= 4 \Big (\sum_{\mu\in \bZ^2}
|c_\mu|^2 \Big )^2.
$$
Gathering the equalities and inequalities above, we   obtain 
$$
\Big\|\sum_{2\pi |\omega|=r}c_\omega \chi_\omega \Big\|_{L^{4}(\bT^2)}^4 =
\|(\Gamma_{\rho,r})_{\rho\in \bZ^2} \|^2_{\ell^{2}(\bZ^2)} \leq (1+4  )\|(c_\omega)\|^4_{\ell^{2}}
= 5\|(c_\omega)\|^4_{\ell^{2}}, 
$$ showing \eqref{zyg2ddual}.

   \subsubsection{Restriction problems in the sub-elliptic framework} 
   Restriction problems have been studied very sparsely in the sub-elliptic setting and are limited to the framework  of certain stratified Lie groups, such as the Heisenberg group,  see for instance \cite{BBG, BF, CC, LW11, Muller} and the references therein.  Note in particular the surprising result of M\"uller  \cite{Muller} stressing
that for  the Heisenberg group there are no non-trivial solutions for Stein's problem in the $L^p$-spaces, for $p > 1$, 
whereas a positive result can be obtained in anisotropic Lebesgue spaces! 

The reason why only a few cases of non commutative Lie groups have been studied is not solely due to the complexity of the group Fourier transform, but   also because several phenomena that occur in the Euclidean case are attenuated or fail  in the case of these groups.  This is for instance the case of  the property of unique continuation, dispersion phenomena, Strichartz estimates and observability; see among others \cite{bgx, BBG, Bahouri86, BF, Benedetto, BL, BSun, FKL, Letrouit, Muller} and the references therein.     
 
  \smallskip 
   
   As mentioned above, the manifold \(M\) is not itself a group and therefore
does not carry an intrinsic Fourier transform. We instead approach restriction
from a spectral viewpoint, in the spirit of \cite{BF,Sogge2017}, but in a
discrete sub-elliptic setting. The harmonic analysis of the Heisenberg group $\bH_1$
yields the relevant decomposition of \(L^2(M)\); within this decomposition, we
choose an orthonormal basis of eigenfunctions for \(\cL_M\). As in Zygmund's
proof on \(\mathbb T^2\), this basis has arithmetic properties that are
essential to the argument.

\subsection{Layout and notation}
The paper is organised as follows. In Section~\ref{def-prop}, we set the notation and recall the basic properties of the Heisenberg group and the associated nilmanifold $M=\Gamma\backslash\bH_1$. This naturally leads to an important decomposition of $L^2(M)$. The paper is written to be almost entirely self-contained: all relevant background is recalled in detail, with some proofs deferred to the appendix. 

Section~\ref{Proof-main} is devoted to the proof of the main result. We first describe the spectral decomposition of $\cL_M$, then establish special properties of the orthonormal basis used in the argument, which extend classical properties of the Zak transform (see Section~\ref{subsec_pfThmcorcorlem_idZ}). The proof of the main theorem then follows. 
We conclude by proving sharpness; the argument here is guided by the natural appearance of the discrete short-time Fourier transform in the relevant expressions, which provided key intuition.

The appendix collects background material on Hermite and Laguerre functions, as well as on nilmanifolds with a particular emphasis on the Heisenberg nilmanifold. In particular, it contains the proof of the sub-elliptic Bernstein inequality on the Heisenberg nilmanifold $M$ stated in Proposition~\ref{proptrivial}.

\smallskip  We conclude this introduction with a comment on notation. In this paper, the letter $C$ will be used to denote  universal constants
which may vary from line to line. If we need the implied constant to depend on parameters, we shall indicate this by subscripts. We  also use the notation $A\lesssim B$  (respectively $A\gtrsim B$) to
denote   bounds of the form $A\leq C B$ (respectively $A \geq C B$), 
  and $A \lesssim_\epsilon B$ 
for $A\leq C_\epsilon B$, where~$C_\epsilon$ depends only  on $\epsilon$.

\medskip

\bigbreak\noindent{\bf Acknowledgments.}
  The authors wish to thank    Philippe Jaming       for fruitful exchanges about the Zak transform.


\section{Notation and basic properties}
\label{def-prop}
\subsection{The Heisenberg group and its standard nilmanifold}
The Heisenberg group   arises in a number of
fundamental domains of mathematics and physics. Owing to
this wide variety of topics, there are different ways of realising it. Whatever the chosen realisation is, the Heisenberg group is a non commutative Lie group, that is a group
endowed with a smooth manifold structure, in which the group operations of multiplication
and inversion are smooth.

\smallskip Throughout this paper, we shall   consider the 3-dimensional Heisenberg group $\bH_1$ realised as the set of
 upper triangular and unipotent real $3 \times 3$ matrices, that is, the matrices of the form \eqref{rea-H1}.
Then $\bH_1$ is the connected simply connected nilpotent Lie group identified with 
$\{(a,b,c)\in 
\bR^3\}$
 equipped with the non commutative   group law: 
\begin{equation} \label{law}
(a,b,c)(a',b',c') =(a+a',b+b',c+c'+ab').
\end{equation}
Its neutral element is the origin 
$(0, 0,0)$, the inverse of a generic element $(a,b,c)$ of $\bH_1$ is just~$(-a, -b, -c+ab)$.  

\medskip We consider its  discrete  subgroup
$$
\Gamma := \{(a,b,c)\in \bH_1 \ : \ a,b,c\in \bZ\}.
$$
The  standard three-dimensional Heisenberg nilmanifold  is then the  right quotient
$$
M:=\Gamma \backslash \bH_1.
$$
Its elements are the $\Gamma$-right equivalence classes over $\bH_1$:
$$
M\ni \dot x = \Gamma x, \quad x\in \bH_1.
$$

We fix the Haar measure
 $dx=dadbdc$ on $\bH_1$  and the counting measure $\sum_{x \in \Gamma} \delta_x$ on $\Gamma$.
 This induces a unique measure $d\dot x$ on $M$, 
 see Appendix \ref{subsec_nilmanifolds}.
Moreover, a fundamental domain in~$\bH_1$ is given by 
$$
\{(a,b,c)\in \bH_1 : a,b,c \in [0,1)\}\sim [0,1)^3;
$$
 on this fundamental domain, the measure induced on $M$ identifies with $dadbdc$. In particular,~$M$ is of volume $\vol (M)= 1$.

\smallskip

Having fixed the Haar measure on $\bH_1$ as above allows us to identify the functional spaces on~$\bH_1$ and on $\bR^3$, see Appendix \ref{subsecLiegr}; for instance, we will write $\cS(\bH_1)$ for the space of Schwartz functions,~$C_c(\bH_1)$ for the space of continuous functions with compact support,~$C_c^\infty(\bH_1)$ for the space of smooth functions with compact support,~$\cS'(\bH_1)$, $\cD'(\bH_1)$,~$L^p(\bH_1)$, $L^p_{loc}(\bH_1)$ etc. 

\smallskip
Similarly fixing a measure on $M$ enables  us to consider the space $L^p(M)$, $p\in [1,\infty]$. 
Moreover, since $M$ is a smooth manifold, the identification above 
implies that the space~$C^\infty(M)$ of smooth functions on $M$ is dense in $L^p(M)$ for any $p\in [1,\infty]$.

\subsection{The Lie algebra of  $\bH_1$} 
  The Lie algebra  $\fh_1$   of  the Heisenberg group $\bH_1$ is
the space of its left-invariant vector fields.   
Given our realisation of the product law in \eqref{law}, 
the canonical basis of $\fh_1$ is   given by:
$$
A := \partial_a, 
\qquad 
B:= \partial_b + a\partial_c, 
\qquad 
S:=\partial_c.
$$
We  readily check that they are invariant under left-translation and  satisfy the canonical commutation relations, known as CCR in quantum mechanics: 
$$
[A,S]=0=[B,S] \quad\mbox{and}\quad
[A,B]=	S .
$$

\smallskip

Note that  $\bH_1$ can be globally reconstructed from its Lie algebra $\fh_1$ by means of the exponential map $\exp:\fh_1\to \bH_1$, which  is a smooth bijection with a smooth inverse, and   the   Baker-Cambell-Hausdorff  formula.  Under the realisation \eqref{rea-H1},    $\fh_1$ may be realised   as the three-dimensional vector space of matrices of the form 
$$
\left(
\begin{array}{ccc}
	0&     a&c\\0&0&b\\0&0&0
\end{array}	\right) , \qquad a,b,c\in \bR,
$$
and the exponential map is therefore nothing else than the matrix exponential $\exp:\fh_1\to \bH_1$. 

\smallskip

As the vector fields $A,B,S$ are left-invariant,  they descend to vector fields $A_M$, $B_M$, $S_M$  etc. on $M$;  see Section \ref{subsubsec_opGM}.
 The vector fields $A_M, B_M, S_M$ form a frame on~$M$
and satisfy  also the canonical commutation relations.  

It is  natural to define the sub-Laplacians on $\bH_1$ and on  $M$
$$
\cL:=-(A^2+B^2) \quad\text{and}\quad
\cL_M := - (A_M^2+ B_M^2).
$$
These linear differential operators are positive, hypoelliptic and essentially self-adjoint on~$C_c^\infty (\bH_1) \subset L^2(\bH_1)$ and on~$C^\infty (M) \subset L^2(M)$ respectively.
We will keep the same notation for their self-adjoint extensions respectively on~$L^2(\bH_1)$ and on~$L^2(M)$.

\subsection{The Schr\"odinger representations $\pi_\lambda$, $\lambda\neq 0$}
 \label{subsec_schro}
 
In this paper, we realise  the Schr\"odinger representations $\pi_\lambda$ of $\bH_1$ for $\lambda\in \bR\setminus \{0\}$ acting on $L^2(\bR)$ as    
\begin{equation}   \label{repHe}
\pi_\lambda(a,b,c) h (u) = e^{2\pi i \lambda (c + ub)} h(u + a), 
\quad (a,b,c)\in \bH_1, \ h\in L^2(\bR), \ u\in \bR.
\end{equation}   

\smallskip 
Recall  that $\pi_\lambda$ is a group homomorphism between $\bH_1$ and the   unitary group~$\cU(L^2(\R))$ of~$L^2(\R)$, and plays the same role as the characters for the Euclidean space,  as   regards  the  Fourier transform on $\bH_1$. For   an overview on this subject, we refer for instance the reader to~\cite{bcdh, CorwinGreenleaf, fermfischer1, Fischerbook, kirillov} and the references therein.  

\medskip 

All along this paper, we keep the same notation for the infinitesimal Schr\"odinger representation on the Lie algebra $\fh_1$ of $\bH_1$. 
We will identify $\fh_1$ with the space of left-invariant vector fields via
$$
X\phi (g)  = \partial_t \phi (g \exp tX), \qquad \phi\in C^\infty(\bH_1), \ g\in \bH_1;
$$
above we have kept the same notation for the matrix $X\in \fh_1$ and the corresponding vector field.
The infinitesimal Schr\"odinger representation on the Lie algebra $\fh_1$ of $\bH_1$ is then defined as 
$$
\pi_\lambda (X) =\partial_{t=0}\pi_\lambda (\exp t X).
$$
We then easily compute
\begin{equation}   \label{rep-al}
\pi_\lambda(A)h(u)= h'(u), 
\quad \pi_\lambda(B)h(u)= 2\pi i \lambda  u h(u), 
\quad \pi_\lambda(S)=2\pi i \lambda \id.
\end{equation}   

\smallskip 
All along this paper, we also keep the same notation for the Schr\"odinger representation on the universal enveloping algebra of  $\fh_1$.
In particular, the Schr\"odinger representation for the sub-Laplacian is the rescaled harmonic oscillator:
\begin{equation} \label{eq:actsub}
\pi_\lambda (\cL) =  -\frac {d^2}{du^2} + (2\pi \lambda u)^2.
\end{equation}
The definition of 
the Hermite functions
$h_\ell$ is  given in \eqref{hermite}. We  rescale them via
\begin{equation} \label{eq:eigsub}
h_{\ell,\lambda}(u):= (2\pi |\lambda|)^{\frac14} \, 
h_\ell\left (\sqrt{2\pi |\lambda|}  u\right ) , \quad u\in \bR,
\end{equation}
so that they become the eigenfunctions of $\pi_{\lambda}(\cL)$:
 \begin{equation}   \label{actionhermite}
\pi_\lambda (\cL)h_{\ell,\lambda} = 2\pi|\lambda|(2\ell +1) h_{\ell,\lambda}, \quad \ell\in \bN_0.
\end{equation}  
Furthermore, the spectrum of $\pi_{\lambda}(\cL)$ is $\{2\pi|\lambda|(2\ell +1), \ell\in \bN_0\}$. 

It is well-known that the matrix coefficients of the Schr\"odinger representation
with respect to the Hermite functions are expressed in terms of Laguerre
functions,  see \eqref{rel}.
With the convention of our paper, this is expressed as 
\begin{equation}
	\label{eq_hllamLag}
(h_{\ell,\lambda} , \pi_\lambda(a,b,c) h_{\ell,\lambda} )
_{L^2(\bR)}
= 
e^{-2\pi i \lambda c}
e^{\pi i \lambda ba}
\cL_\ell \left(\pi|\lambda| (a^2+b^2)\right).	
\end{equation}

   \subsection{Decomposition of $L^2(M)$ under the regular representation}
The Heisenberg group~$\bH_1$ acts unitarily on $L^2(M)$ via the right regular representation:
\begin{equation}   \label{repreg}
R:\bH_1\to \mathscr{U}(L^2(M)), \qquad R(g)f(\dot x) = f(\dot x g), \ g\in \bH_1, \ f\in L^2(M), \ \dot x\in M.
\end{equation} 
Moreover, this action leads to a decomposition of $L^2(M)$ that we briefly present here.

\subsubsection{Statement}
Here, we recall features of this decomposition, especially the properties that will be useful in the rest of the paper. 

\begin{theorem}
\label{thm_decR}
The decomposition in Fourier series in the $c$-variable  yields 
an orthogonal decomposition of 
the space $L^2(M)$ into subspaces 
\begin{equation}   \label{rdecL2}
L^2(M) =\oplus^\perp_{\lambda\in \bZ} L^2_\lambda(M),
\end{equation}  
that respects the action of the regular representation $R$   
 defined by \eqref{repreg};
 the subspaces $L^2_\lambda(M)$ are explicitly described in Section \ref{subsubsec_L2lambda}.

\begin{enumerate}
	\item The space $L^2_0(M)$ admits a further decomposition  
 $$
 L^2_0(M) = \oplus_{w\in \bZ^2}^\perp \bC \chi_\omega, 
 $$
 that is irreducible and respects the action of $R$. 
 Above, $\chi_\omega\in C^\infty(M)$  is defined via 
 \begin{equation}
 	\label{def_chiomega}
 	\chi_\omega (\dot x) := \exp (2\pi i \omega\cdot (a,b)), 
 	\qquad x=(a,b,c)\in \bH_1, \ \omega\in \bZ^{2}.
 \end{equation}
\item Let $\lambda\in \bZ\setminus\{0\}$. 
In order to obtain the further decomposition of $L^2_\lambda(M)$, 
we define the generalised Berezin-Weil-Zak transform as
\begin{equation}
	\label{eq:BWZ}
		\BWZ_{\lambda,q}(h)(\dot x) :=\frac{\sqrt{|\lambda|}}{\lambda}  \sum_{k\in \bZ} 
e^{-2\pi i q \frac {k} \lambda  }
e^{2\pi i  (\lambda c + k b)} h\left (\frac {k}\lambda + a\right ), 
\end{equation}
for $h\in \cS(\bR)$, $q\in \bZ$ and  $x=(a,b,c)\in \bH_1$. 
The BWZ transform enjoys the following properties:
\begin{enumerate}
\item For any $q\in \bZ$ and $h\in \cS(\bR)$, $\BWZ_{\lambda,q}(h)$ is a well defined smooth function on $M$ which belongs to  $L^2_\lambda(M)$. It satisfies 
for any~$k_1\in \bZ$, 
$$
\BWZ_{\lambda,\lambda k_1 +q}(h)=\BWZ_{\lambda,q}(h), 
$$
allowing us to see the parameter $q$ in $\bZ/\lambda\bZ$.
\item 	For each $q\in \bZ/\lambda\bZ$, 
	the transformation
	$\BWZ_{\lambda,q}:\cS(\bR)\to C^\infty(M)$ extends uniquely  to a unitary linear transformation 
	$$
	\BWZ_{\lambda,q}:L^2(\bR)\longrightarrow L^2_\lambda(M).
	$$
	\item For any $q_0,q\in \bZ$ and $\Gamma (a,b,c)=\dot x\in M$, we have
\begin{equation}\begin{aligned}\label{action-osc}
\BWZ_{\lambda,q_0}(h) \left (\dot x \left (\frac {q} \lambda,0,0\right)\right)
&= e^{-2\pi i   q b} e^{2\pi i q_0 \frac  q  \lambda} \BWZ_{\lambda,q_0}(h) (\dot x),
\\
\BWZ_{\lambda,q_0}(h) \left (\dot x \left (0, \frac {q} \lambda,0\right)\right)
&= e^{2\pi i   q a} \BWZ_{\lambda,q_0 -q}(h) (\dot x).\end{aligned}
\end{equation} 

\end{enumerate}
\item 
The subspaces 
$$
L^2_{\lambda,q}(M) := \BWZ_{\lambda,q}(L^2(\bR)), \qquad q\in \bZ/\lambda\bZ,
$$
of $L^2_\lambda(M)$ are orthogonal and their sums fills:
$$
L^2_\lambda(M) = \oplus^\perp_{q\in \bZ/\lambda\bZ}L^2_{\lambda,q}(M).
$$
Moreover, the regular representation $R$ defined by   \eqref{repreg} acts on $L^2_{\lambda,q}(M)$, 
where it is unitarily equivalent to the (irreducible) Schr\"odinger representation $\pi_\lambda$
with~$\BWZ_{\lambda,q}$ as  intertwiner.
\end{enumerate}
\end{theorem}

Theorem \ref{thm_decR}  is well-known with perhaps different conventions, but we include the proofs for the sake of completeness:
the decomposition in \eqref{rdecL2} and
 Part (1) are shown below (in Sections \ref{subsubsec_L2lambda} and \ref{subsubsec_lambda=0} respectively), while the proofs of  Parts (2) and (3)  are postponed to Appendix \ref{sec_BWZ}.
 References for this material include \cite{KTX, Th2009}.

\subsubsection{The BWZ transform}

The generalised Berezin-Weil-Zak transform \eqref{eq:BWZ}  plays a key role for  our purpose.  Such transform   is well-known with different scales and up to unitary operators.   According to the field of use, it  is  also referred to  as the  Bloch-Floquet, Gabor, Gel'fand or Zak transform. It  was first introduced by Gel'fand \cite{Gelfand} in~1950   for the sake of a problem in ordinary differential
equations, then extended   by Weil  \cite{Weil} in~1964 to the setting of locally compact abelian groups with respect to arbitrary
closed sub-groups, and  rediscovered in~1967 by Zak  \cite{Zak}   who used it to construct a quantum mechanical representation
for the description of the dynamic  of an electron in the presence of a constant 
magnetic or electric field.  
These transforms  have  been a matter of great interest in the last decades  and   developed by many authors on locally compact  groups (abelian and non commutative) and groupoids, see for instance~\cite{DJ, Folland, Gro, jamain, Janssen,  Kut,  Zak} and the references therein.  
 
We observe that the orthogonality of the~$
	\BWZ_{\lambda,q}(L^2(\bR))=L^2_{\lambda,q}(M)$, $q\in \bZ/\lambda\bZ$, 
	together with the unitarity of $\BWZ_{\lambda,q}$ are equivalent to 
 \begin{equation}
\label{orthBVZ}
(\BWZ_{\lambda_1,q_1}(h_1)	, \BWZ_{\lambda_2,q_2}(h_2))_{L^2(M)}
=\delta_{\lambda_1=\lambda_2,q_1=q_2} (h_1,h_2)_{L^2(\bR)},
\end{equation}
holding for all $\lambda_1,\lambda_2\in \bZ\setminus\{0\}$, $h_1,h_2\in \cS(\bR)$ and $q_1,q_2\in \bZ$;
this explains our normalisation for $\BWZ_{\lambda,q}$. 
The $\BWZ_{\lambda,q}$ transform  intertwining the regular and Schr\"odinger representations means that 
	\begin{equation}
\label{acR}
	\forall g\in \bH_1, \qquad 
	\BWZ_{\lambda,q}\pi_\lambda(g) = R(g) \BWZ_{\lambda,q},
	\end{equation}
so
	\begin{equation}\label{action-Vf}
	  \forall X\in \fh_1,\qquad 
	\BWZ_{\lambda,q}\pi_\lambda(X) = X_M \BWZ_{\lambda,q}.
	\end{equation}

\subsubsection{A first decomposition}
\label{subsubsec_L2lambda}
We start by resorting to Fourier series in the central variable: 
$$
f = \sum_{\lambda\in \bZ} f_\lambda,
$$
where $f_\lambda$ is loosely defined as
\begin{equation}
	\label{eq_flambda(a,b,c)}
f_\lambda(a,b,c) :=  e^{2\pi i \lambda c}
 \int_0^1 f(a,b,c')e^{-2\pi i \lambda c'} dc',  \quad  (a,b,c) \in  [0,1)^3\sim M.
\end{equation}
 Let us pause  to give a meaning to this formula.
Since $f\in L^2(M)$,  we know $\phi:=f_{\bH_1}\in L^2_{loc}(\bH_1)$ so the function defined via
$$
\phi_\lambda (a,b,c) :=  e^{2\pi i \lambda c}
 \int_0^1 f(a,b,c')e^{-2\pi i \lambda c'} dc', \qquad (a,b,c)\in \bH_1,
 $$
 is also in $L^2_{loc}(\bH_1)$; we check that it is $\Gamma$-invariant, and we set $f_\lambda = (\phi_\lambda)_M$. 
 
 \smallskip
In fact, we check  that each $f_\lambda$ above is in the subspace
\begin{equation}   \label{defsub}
  L^2_\lambda (M) := \{f\in L^2(M) \ \colon \ f(\Gamma( a,b,c)) = f(\Gamma(a,b,0))e^{2\pi i \lambda c}
\}
\end{equation}   
 of $L^2(M)$, and is isomorphic via $f\mapsto f_{\bH_1}=\phi$ onto the space $L^2_\lambda (\bH_1)$ of function $\phi \in L^2_{loc}(\bH_1)$ that are $\Gamma$-invariant and satisfy 
$$
 \phi(a,b,c) = \phi(a,b,0)e^{2\pi i \lambda c} \ \mbox{for all} \ (a,b,c)\in \bH_1.
 $$
We have obtained the first decomposition announced in \eqref{rdecL2}, that is, 
$$
L^2(M) =\oplus^\perp_{\lambda\in \bZ} L^2_\lambda(M).
$$
The orthogonal projectors of this decomposition 
are denoted by 
 \begin{equation}
	\label{eqdef_Prlambda} 
	\Pr_\lambda:L^2(M)\to L^2_\lambda(M), \quad \lambda\in \bZ.
\end{equation}
They enjoy the following properties:
\begin{lemma}
  \label{lem_Prlambda}
  Let $\lambda\in \bZ$.
  The orthogonal projector $P_\lambda$ is formally given via
  $$
  \Pr_\lambda f(\dot x) =e^{2\pi i \lambda c}
 \int_0^1 f(a,b,c')e^{-2\pi i \lambda c'} dc',
 	\qquad x=(a,b,c)\in \bH_1, \quad f\in C^\infty(M).
 	$$
Moreover, for any $f\in C^\infty(M)$, 
we have
$$
\|\Pr_\lambda f\|_{L^p(M)}\leq \|f\|_{L^p(M)}, \quad p\in [1,\infty].
$$
  \end{lemma} 
  \begin{proof}
  	The formula for the projectors follows readily from \eqref{eq_flambda(a,b,c)}.
It implies   by Jensen's inequality that we have for any $f\in C^\infty(M)$
$$
|\Pr_\lambda f(\dot x)|^p
=\Big|\int_0^1 f(a,b,c')e^{-2\pi i \lambda c'} dc'\Big|^p
\leq \int_0^1 |f(a,b,c')e^{-2\pi i \lambda c'}|^p dc'
= \int_0^1 |f(a,b,c')|^p dc',
$$
hence, 
$$
\|\Pr_\lambda f\|_{L^p(M)}^p \leq \int_{[0,1)^3}\int_0^1
|f(a,b,c')|^p dc' dadbdc=
\int_{[0,1)^3}
|f(a,b,c')|^p dadbdc'=
\|f\|_{L^p(M)}	^p.
$$
The case $p=+\infty$ also holds true by the same formula.
  \end{proof}

We observe that Decomposition \eqref{rdecL2}  respects the action of the regular representation $R$ defined by \eqref{repreg}:
\begin{lemma}
\label{lem_RactsL2lambda}
Let $\lambda\in \bZ$.
For any $g\in \bH_1$, 
	$R(g)$ maps $L^2_\lambda(M)$ into itself.
	Moreover, we have for any $f\in L^2_\lambda(M)$
 $$
 \forall c\in \bR,\quad 
 R(0,0,c) f =e^{2\pi i \lambda c} f.
 $$  
\end{lemma}
\begin{proof}
We  compute for any $\phi:\bH_1\to \bC$ 
 $$
 \forall (a',b',c'),(a,b,c) \in \bH_1\qquad 
 \phi((a',b',c')(a,b,c)) =\phi(a'+a,b'+b,c'+c+a'b),
 $$
 and we check that    if $\phi\in L^2_\lambda(\bH_1)$, then so does 
 $(a',b',c')\mapsto \phi((a',b',c')(a,b,c))$.
 This shows that $R(g)$ maps $L^2_\lambda(M)$ into itself as well as the action of $R(0,0,c)$.
\end{proof}

To understand the action of $\bH_1$ via $R$ onto each $L^2_\lambda(M)$, we need to distinguish the cases~$\lambda=0$ and $\lambda\neq 0$.

\subsubsection{Case $\lambda=0$}
\label{subsubsec_lambda=0}
 The construction above  implies the following property for $L^2_0(M)$:
  \begin{lemma}
  \label{lem_L20identL2T2}
  The space  $L^2_0(M)$ is isomorphic via $f\mapsto f|_{c=0}$ to the space~$L^2(\bT^2)$.
  \end{lemma}  
  \begin{proof} 
  By construction, 
  $L^2_0 (\bH_1)$ coincides with the space of functions $\phi\in L^2_{loc}(\bH_1)$ which do not depend on the central variable, and as functions of $(a,b)\in \bR^2$ which are $\bZ^2$-periodic. 
Hence $f\mapsto f|_{c=0}$ is a linear map $C^\infty(M)\cap L^2_0(M) \to C^\infty(\bT^2)$ that extends continuously uniquely into an isomorphism between $L^2_0(M)$ and $L^2(\bT^2)$.
  \end{proof}
  
 We obtain a further decomposition of $L^2_0(M)$ by considering the Fourier series on $\bT^2$. 
 For this, 
 we define the function $\chi_\omega\in C^\infty(M)$ via \eqref{def_chiomega}, or in other words via the $\Gamma$-invariant function~$(a,b,c)\mapsto e^{2i\pi (\omega_1 a + \omega_2 b)}$ on $\bH_1$
 for each $\omega=(\omega_1,\omega_2)\in \bZ^2$.

\smallskip
 Any $f\in L^2_0(M)$ can then be written as 
 $$
 f = \sum_{\omega\in \bZ^2} (f,\chi_\omega)_{L^2(M)}\chi_\omega.
 $$
We check readily that 
$$
R(a,b,c)\chi_\omega =e^{ 2\pi i  (\omega_1 a + \omega_2 b)}
\chi_\omega, \quad (a,b,c)\in \bH_1.
$$ 
Therefore, $R$ acts on $\bC \chi_\omega$ and we have an orthogonal decomposition 
$$
L^2_0 (M) =\oplus^\perp_{\omega\in \bZ^2}  \bC \chi_\omega.
$$
This shows Part (1) of Theorem \ref{thm_decR}.

 \subsubsection{Case $\lambda\neq 0$}
 \label{subsec_constfL2lambda}
\label{sec_cor_lem_phipinuh}

The decomposition for $\lambda\neq 0$ is also classical in harmonic analysis,
but its explicit construction is more technical.
For completeness, we recall it  in Appendix~\ref{sec_BWZ}, thereby
completing the proof of Theorem~\ref{thm_decR}.

\section{Proof of the main results}\label{Proof-main}

\subsection{The spectral decomposition of $\cL_M$}
\label{spectral}

We consider the functions~$\chi_\omega$, $\omega\in \bZ^2$, defined in \eqref{def_chiomega}, as well as  
$$
h_{\lambda,q,\ell} := \BWZ_{\lambda,q} h_{\ell,\lambda}, 
\quad \lambda\in \bZ\setminus\{0\}, \ q\in \bZ/\lambda\bZ, \ \ell \in \bN_0,
$$
with $h_{\ell,\lambda}$ the normalised Hermite functions defined by \eqref{eq:eigsub}, and where 
the $\BWZ_{\lambda,q}$ transform was defined in   Theorem \ref{thm_decR}. Then
\begin{equation}
	\label{eq_hlambdaqell}
	h_{\lambda, q,\ell}\left(\Gamma (a,b,c)\right)=\frac{\sqrt{|\lambda|}}{\lambda}  \sum_{k\in \bZ} 
e^{-2\pi i q \frac {k} \lambda  }
e^{2\pi i  (\lambda c + k b)} h_{\ell,\lambda} \left(\frac {k}\lambda + a\right).
\end{equation}

\begin{theorem}
\label{thm_spdec}
The functions~$\chi_\omega$ and~$h_{\lambda,q,\ell}$ on $M$ are smooth and  $L^2$-normalised:
\begin{equation}
	\label{eq_norm}
\|\chi_\omega\|_{L^2(M)}=1
\quad\mbox{and}\quad
\|h_{\lambda,q,\ell}\|_{L^2(M)}=1.
\end{equation}
They are  eigenfunctions of the sub-Laplacian $\cL_M$ on $M$  
\begin{align*}
	\cL_M h_{\lambda,\ell, q} &= 2\pi|\lambda|(2\ell +1)h_{\lambda,\ell, q},\quad \lambda\in \bZ\setminus\{0\}, \ q\in \bZ/\lambda\bZ, \ \ell\in \bN_0,\\
	\cL_M \chi_\omega &= (2\pi)^2 |\omega|^2 \chi_\omega, \qquad  \qquad \omega\in \bZ^2, 
\end{align*}  and form an orthonormal basis of $L^2(M)=\oplus^\perp_{\lambda\in \bZ} L^2_\lambda(M)$. 
More precisely, $(\chi_\omega)_{\omega\in \bZ^2}$ is an orthonormal basis of $L^2_0(M)$
while,  for each $\lambda\in \bZ\setminus\{0\}$, $(h_{\lambda,q,\ell})_{q\in \bZ/\lambda\bZ,\ell \in \bN_0}$
is an orthonormal basis of $L^2_\lambda(M)$.
\end{theorem}
\begin{proof}
The properties for $\chi_\omega$  follow  from Fourier analysis on the two-dimensional torus, while the ones for $h_{\lambda,\ell, q}$ follows from Theorem \ref{thm_decR}  and the properties of $h_{\ell,\lambda}$ recalled in Section \ref{subsec_schro}.
Indeed, by virtue of \eqref{action-Vf}, one has  
\begin{align*}\cL_M h_{\lambda,q,\ell}&= \cL_M \BWZ_{\lambda,q} h_{\ell,\lambda}= \BWZ_{\lambda,q} \pi_\lambda(\cL) h_{\ell,\lambda},
\end{align*}
which gives the result thanks to \eqref{eq:actsub} and \eqref{actionhermite}. 
\end{proof}

\begin{remark}
The spectral decomposition of $\cL_M$ is well-known.
For instance, Deninger and Singhof in \cite{DeningerSinghof} described an orthogonal basis of $\cL_M$-eigenfunctions of $L^2(M)$ in \cite{DeningerSinghof}.	
	With our notation and after $L^2$-normalisation, their eigenfunctions are the smooth and $L^2$-normalised functions $\chi_\omega$, $\omega\in \bZ^2$, as above, and 
$g_{\lambda,r,\ell}$, $r\in \bZ/\lambda\bZ$,  $\lambda\in \bZ\setminus\{0\}$, $\ell\in \bN_0$,
\[
 g_{\lambda, r, \ell}\left(\Gamma (a,b,c)\right)=
e^{2\pi i( \lambda c +rb)}\sum_{k_1\in \bZ}e^{2\pi i\lambda k_1 b} h_{\ell, \lambda} \left (a+\frac r{\lambda}+k_1\right).
\]
Indeed, we have
\begin{align*}
	\cL_M g_{\lambda,r,\ell} &= 2\pi|\lambda|(2\ell +1)h_{\lambda,r,\ell},\qquad \lambda\in \bZ\setminus\{0\}, \ r\in \bZ/\lambda\bZ, \ \ell\in \bN_0,\\
	\cL_M \chi_\omega &= (2\pi)^2 |\omega|^2 \chi_\omega, \qquad \omega\in \bZ^2.
\end{align*}
This basis is unitarily equivalent to ours since we have  
for any $\lambda\in \bZ\setminus\{0\}$ and $\ell\in \bN_0$,
	\begin{align*}
		h_{\lambda, q,\ell}&=\frac{\sqrt{|\lambda|}}{\lambda} 
\sum_{0\leq r<|\lambda|}e^{-2\pi i q \frac {r} \lambda}
g_{\lambda,r,\ell},\quad q\in \bZ / \lambda \bZ,\\
g_{\lambda, r_0, \ell} &= \frac{\sqrt{|\lambda|}}{\lambda}  
\sum_{0\leq q<|\lambda|}e^{2\pi i q \frac {r_0} \lambda}
h_{\lambda, q, \ell}, \quad r_0\in \bZ / \lambda \bZ.
	\end{align*}
\end{remark}

\subsection{A key property of the $\cL_M$-eigenfunctions $h_{\lambda,q,\ell}$}
\label{subsec_pfThmcorcorlem_idZ}

As emphasised in the proof of Zygmund's estimate in Section \ref{sketch},
the main difficulty in the proof of Theorem~\ref{main} lies in estimating the
$L^4$ norms related to the functions $h_{\lambda,q,\ell}$.
The starting point for these estimates is the following new identity:

\begin{theorem}
\label{thm_keyprop}
Let $\lambda\in \bZ\setminus\{0\}$ 
and $\ell\in \bN_0$. For any  family of coefficients $\gamma_q$,
$q\in \bZ/\lambda\bZ$, we have
\begin{equation}\label{estL4}
	\bigg\|\sum_{q\in \bZ/\lambda\bZ} \gamma_q h_{\lambda,q,\ell}
	\bigg\|_{L^4(M)}^4
	=	
	\sum_{a,b\in \bZ} \Bigg|
	\Big (\sum_{q\in \bZ/\lambda\bZ } \gamma_{q}\overline{\gamma_{q-b}}e^{2\pi i q\frac {a}\lambda}	\Big )
	 \cL_\ell \left (\pi \frac{a^2+b^2}{|\lambda|}\right)\Bigg|^2,
	\end{equation}	
	where $\cL_\ell$ is the Laguerre function of type 0 and degree $\ell\in \bN_0$.
	\end{theorem} 
	
To prove  Theorem \ref{thm_keyprop}, we start by pointing out that  the 
traditional Zak transform $Z$  coincide with our Berezin-Weil-Zak transform for $\lambda=1$, $q=0$,  restricted to a part of the fundamental domain, namely $[0,1)^2\times \{0\}$, that is, 
$$
Zh(a,b) := \BWZ_{1,0} h  (a,b,0)
= \sum_{k\in \bZ} 
e^{2\pi i   k b} h\left (k+ a\right ),
\qquad (a,b)\in [0,1)^2, \ h\in \cS(\bR);
$$
The set 
$[0,1)^2$ is often identified with the two-dimensional torus~$\bT^2 = \bR^2/\bZ^2$.
Well-known identities for the Zak transform 
\cite{Folland} (see also \cite{FaulhuberSZ})
generalise to $\BWZ_{\lambda,q}$, for instance: 

\begin{proposition}
	\label{prop_idZ}
Let $\lambda\in \bZ\setminus\{0\}$. For any $h\in \cS(\bR)$ and any family of coefficients $\gamma_q$,
$q\in \bZ/\lambda\bZ$, we have
$$
	\bigg  \|\sum_{q\in \bZ/\lambda\bZ} \gamma_q\BWZ_{\lambda,q} h
	\bigg \|_{L^4(M)}^4
	=	
	\sum_{a,b\in \bZ} \bigg|
	\big (\sum_{q\in \bZ/\lambda\bZ } \gamma_{q}\overline{\gamma_{q-b}}e^{2\pi i q\frac {a}\lambda}	\Big )
	 \left(h,  \pi_\lambda\left (\frac {a}{\lambda} ,\frac{b}\lambda,0\right ) h\right)_{L^2(\bR)}\bigg|^2.
	$$	
	\end{proposition}

\begin{proof}[Proof of Proposition \ref{prop_idZ}]
We observe
\begin{align*}
	\bigg \|\sum_{q\in \bZ/\lambda\bZ} \gamma_q\BWZ_{\lambda,q} h
	\bigg \|_{L^4(M)}^4& = 
	\bigg \|\sum_{q_1,q_2\in \bZ/\lambda\bZ} \gamma_{q_1} \overline{\gamma_{q_2}} \BWZ_{\lambda,q_1} h\overline{\BWZ_{\lambda,q_2} h} \bigg\|_{L^2(M)}^2 \\
	& = 
	\int_{[0,1]^2}  |F(a,b)|^2  dadb,
	\end{align*}
with 
\[
F(a,b):=\sum_{q_1,q_2\in \bZ/\lambda\bZ} \gamma_{q_1} \overline{\gamma_{q_2}} (\BWZ_{\lambda,q_1} h\overline{\BWZ_{\lambda,q_2} h})(a,b,0).
\]
We now expand $F\in L^2([0,1]^2)$ in Fourier series and obtain by the Plancherel formula:
\begin{align*}
\int_{[0,1]^2}  |F(a,b)|^2 (a,b,0) dadb
&=	
\sum_{a',b'\in \bZ} 
	\bigg| \int_{[0,1]^2} F(a,b) e^{2\pi i (a'b-ab')} da db \bigg|^2.
\end{align*}
Hence, we have obtained:
\begin{align}
	&\bigg \|\sum_{q\in \bZ/\lambda\bZ} \gamma_q\BWZ_{\lambda,q} h
	\bigg \|_{L^4(M)}^4\nonumber
	\\	
	&\quad = \sum_{a',b'\in \bZ} 
	\bigg|\sum_{q_1,q_2\in \bZ/\lambda\bZ} \gamma_{q_1} \overline{\gamma_{q_2}} \int_{[0,1]^2} (\BWZ_{\lambda,q_1} h \ \overline{\BWZ_{\lambda,q_2} h})(a,b,0) \ e^{2\pi i (a'b-ab')} da db \bigg|^2.
	\label{eq_pfcorlem_idZ}
\end{align}

In order to analyse each integral over $[0,1]^2$ in the right-hand side of \eqref{eq_pfcorlem_idZ}, we start with noticing   that for any $x'=(a',b',c')\in \bH_1$, the integral
\begin{equation}
\label{eq_pflem_idZ}
\int_M \BWZ_{\lambda,q_1} h_1(\dot x)\overline{\BWZ_{\lambda,q_2} h_2(\dot xx')}d\dot x	
\end{equation}
is equal to 0 if $q_1\neq q_2$ in $\bZ/\lambda \bZ$ since
		$L^2_{\lambda,q_1}\perp L^2_{\lambda,q_2}$ when $q_1\neq q_2$ (see Theorem \ref{thm_decR}), while if~$q_1= q_2$ in $\bZ/\lambda \bZ$, 
then, by virtue of \eqref{orthBVZ}-\eqref{acR} and \eqref{repHe},  it is equal to 
$$
 \left(h_1,  \pi_\lambda(x') h_2\right)_{L^2(\bR)}=
		\int_\bR  h_1(u) \overline{\pi_\lambda(x') h_2(u)}du\\
		= \int_\bR h_1(u) \overline{ h_2(u+a')}
		e^{-2\pi i \lambda(c'+ub')}du, 
		$$
and in this case, it is independent of $q_1=q_2$ in $\bZ/\lambda \bZ$.

According to  \eqref{action-osc}, for
\begin{equation}
	\label{eq_x'lem_propBWZ1}
	x'=\left (\frac {a'}\lambda,0,0\right)\left (0,\frac {b'}\lambda,0\right )
=\left (\frac {a'}\lambda,\frac {b'}\lambda,\frac {a'b'}{\lambda^2}\right ),\qquad a',b'\in \bZ,
\end{equation}
the formula in
\eqref{eq_pflem_idZ} is equal to 
\begin{align*}
&\int_{[0,1]^2} \BWZ_{\lambda,q_1} h_1(a,b,0)
\overline{\BWZ_{\lambda,q_2} h_2\left ((a,b,0)x'\right )}dadb	\\
&\quad = 
\int_{[0,1]^2} \BWZ_{\lambda,q_1} h_1( a,b,0)
\overline{ 
e^{2\pi i b' (a+\frac {a'}\lambda)} 
\BWZ_{\lambda,q_2-b'} h_2\big((a,b,0)(\frac {a'}\lambda,0,0)\big)}dadb\\
&\quad = 
\int_{[0,1]^2} \BWZ_{\lambda,q_1} h_1( a,b,0)
\overline{e^{2\pi i b' (a+\frac {a'}\lambda)}  
e^{-2\pi i a'b}e^{2\pi i (q_2-b')\frac {a'}\lambda}
\BWZ_{\lambda,q_2-b'} h_2(a,b,0)}dadb\\
&\quad = 
e^{-2\pi i q_2\frac {a'}\lambda}
\int_{[0,1]^2} \BWZ_{\lambda,q_1} h_1( a,b,0)
e^{2\pi i (a'b-ab')}
\overline{ 
\BWZ_{\lambda,q_2-b'} h_2(a,b,0)}dadb.
\end{align*}
We have therefore obtained that each integral in the right-hand side of \eqref{eq_pfcorlem_idZ} is 
\begin{align*}
&\int_{[0,1]^2} \BWZ_{\lambda,q_1} h_1( a,b,0)
e^{2\pi i (a'b-ab')}
\overline{ 
\BWZ_{\lambda,q_2} h_2(a,b,0)}dadb
\\&\quad =e^{2\pi i (q_2+b')\frac {a'}\lambda}	
\int_{M} \BWZ_{\lambda,q_1} h_1(\dot x)
\overline{\BWZ_{\lambda,q_2+b'} h_2(\dot x x')}d\dot x
\\&\quad=\begin{cases}
	0 &\text{if}\ q_1 \neq q_2 +b' \ \text{mod} \ \lambda\\
	e^{2\pi i q_1\frac {a'}\lambda}	
\left(h,  \pi_\lambda(x') h\right)_{L^2(\bR)} 
&\text{otherwise},
\end{cases}
\end{align*}
with $x'$ as in \eqref{eq_x'lem_propBWZ1}.
Going back to \eqref{eq_pfcorlem_idZ}, 
we obtain
\begin{align*}
&\bigg \|\sum_{q\in \bZ/\lambda\bZ} \gamma_q\BWZ_{\lambda,q} h
	\bigg \|_{L^4(M)}^4\nonumber
	\\	
	&\quad = \sum_{a',b'\in \bZ} 
	\bigg|\sum_{q_1\in \bZ/\lambda\bZ} \gamma_{q_1} \overline{\gamma_{q_1-b'}} 
	e^{2\pi i q_1\frac {a'}\lambda}	
\left(h,  \pi_\lambda\left (\frac {a'}\lambda,\frac {b'}\lambda,\frac {a'b'}{\lambda^2}\right ) h\right)_{L^2(\bR)}
\bigg|^2
\\	
	&\quad = \sum_{a',b'\in \bZ} 
	\bigg|\sum_{q_1\in \bZ/\lambda\bZ} \gamma_{q_1} \overline{\gamma_{q_1-b'}} 
	e^{2\pi i q_1\frac {a'}\lambda}	
\left(h,  \pi_\lambda\left (\frac {a'}\lambda,\frac {b'}\lambda,0\right ) h\right)_{L^2(\bR)}
\bigg|^2,	
\end{align*}
showing the desired formula.
\end{proof}

\begin{proof}[Proof of Theorem \ref{thm_keyprop}]
The proof of Theorem \ref{thm_keyprop} follows readily by applying Proposition \ref{prop_idZ} to $h=h_{\lambda,\ell}$
and then using \eqref{eq_hllamLag}.
\end{proof}

\subsection{Proof of Theorem \ref{main}}\label{Proof-strategy}

\subsubsection{Spectral decomposition of~$\sqrt {\cL_M}$}
In view of   Theorem \ref{thm_spdec}, we know that the  spectrum of~$\sqrt \cL_M$ is the following discrete subset of $[0,+\infty)$:
$$
\spec (\sqrt{\cL_M})=
\left\{ \sqrt{ 2\pi|\lambda|(2\ell +1)} \colon  \lambda\in \bZ\setminus\{0\}, \ q\in \bZ/\lambda\bZ, \ \ell\in \bN_0\right\} \cup \left \{ 2\pi |\omega| \colon  \omega\in \bZ^2\right\}.
$$
Its orthogonal decomposition is described by the orthogonal projection $\Pi_\mu$, $\mu\in \spec (\sqrt{\cL_M})$, given for any $f\in L^2(M)$ by 
\begin{equation}\label{spproj}
\Pi_\mu (f) = \sum_{\substack{
\lambda\in \bZ\setminus\{0\}, q\in \bZ/\lambda\bZ, \ell\in \bN_0\\
\sqrt{ 2\pi|\lambda|(2\ell +1)} = \mu }} 
(f,h_{\lambda,q,\ell})_{L^2(M)} h_{\lambda,q,\ell}
\ + \ \sum_{\substack{\omega\in \bZ^2\\ 
2\pi |\omega| =\mu}}
(f,\chi_\omega)_{L^2(M)}\chi_\omega.
\end{equation}
Then
\begin{equation}\label{spprojdec}
\|\Pi_\mu (f)\|^2_{L^2(M)} = \sum_{\substack{
\lambda\in \bZ\setminus\{0\}, q\in \bZ/\lambda\bZ, \ell\in \bN_0\\
\sqrt{ 2\pi|\lambda|(2\ell +1)} = \mu }} 
|(f,h_{\lambda,q,\ell})_{L^2(M)}|^2
\ + \ \sum_{\substack{\omega\in \bZ^2\\ 
2\pi |\omega| =\mu}}
|(f,\chi_\omega)_{L^2(M)}|^2.
\end{equation}
Note that the above sums given by \eqref{spproj} are always finite  for a given $\mu\in \spec (\sqrt{\cL_M})$, since~$|\lambda|\leq {\mu^2} / {2\pi} $ and~$2\ell +1 \leq \mu^2 / 2\pi$. 
They are empty when $\mu\notin \spec (\sqrt{\cL_M})$; in this case $\Pi_\mu =0$. 
Moreover, the spectral projectors $\Pi_\mu$'s and the orthogonal projectors $\Pr_\lambda$ onto $L^2_\lambda(M)$ (defined in \eqref{eqdef_Prlambda}) commute by construction:
$$ 
\forall \lambda\in \bZ, \ \mu\geq 0, \qquad \Pi_\mu \Pr_\lambda = \Pr_\lambda \Pi_\mu,
$$
with 
$$
\Pi_\mu \Pr_0 f = \Pr_0\Pi_\mu  f=\sum_{\omega\in \bZ^2\colon 
2\pi |\omega| =\mu}
(f,\chi_\omega)_{L^2(M)}\chi_\omega, 
$$
and for $\lambda\neq 0$,
\begin{equation}
\label{eq_PimuPrf}
\Pi_\mu \Pr_\lambda f =  \Pr_\lambda \Pi_\mu f =\sum_{\substack{
 q\in \bZ/\lambda\bZ, \ell\in \bN_0\\
\sqrt{ 2\pi|\lambda|(2\ell +1)} = \mu }}  (f,h_{\lambda,q,\ell})_{L^2(M)} h_{\lambda,q,\ell} .
\end{equation}

\subsubsection{The key property of $\Pi_\mu \Pr_\lambda $}

We have already observed that Zygmund's proof of the restriction theorem on
$\mathbb T^2$~\cite{zygmund} relies on arithmetic properties of the
$L^4$ norms of a particular basis of eigenfunctions, see
Section~\ref{sketch}. In the Heisenberg nilmanifold setting, the analogue of this combined with Theorem \ref{thm_keyprop}
yield
the following statement.
It   will play a crucial role in the proof of
Theorem~\ref{main}:

\begin{proposition}
	\label{propkeyprop} 
	
We fix  $\lambda\in \bZ\setminus\{0\}$ and $\mu\geq 0$. 
We assume that $\Pi_\mu \Pr_{\lambda}\neq 0$.
Consequently,~$\mu$ satisfies $2\pi \mu^2 = |\lambda|(2\ell+1)$ for a (fixed) $\ell\in \bN_0$, and 
we have 
\begin{align*}
\|\Pi_\mu \Pr_\lambda f\|_{L^4(M)} 
&\leq  
\|\Pi_\mu \Pr_\lambda f\|_{L^2(M)} 
\bigg(\sum_{a,b\in \bZ} \left|
		 \cL_\ell \left (\pi \frac{a^2+b^2}{|\lambda|}\right)\right|^2\bigg)^{1/4}\\
		 &\lesssim_\epsilon \mu^{ (1+\epsilon)/2}\|\Pi_\mu \Pr_\lambda f\|_{L^2(M)} ,
		 \end{align*}
		 with an implicit constant depending on $\epsilon\in (0,1)$ but not on $f\in L^2(M)$, $\lambda\in \bZ\setminus\{0\}$ and $\mu$.
\end{proposition}

\begin{proof}[Proof of Proposition \ref{propkeyprop}]
Using \eqref{eq_PimuPrf}, we set $\gamma_{\lambda,q,\ell}= (f,h_{\lambda,q,\ell})_{L^2(M)}$, and we have 
\[
\Pi_\mu \Pr_{\lambda} f= \sum_{
q\in \bZ/\lambda\bZ}
\gamma_{\lambda,q,\ell} h_{\lambda,q,\ell}.
\]
The first inequality follows from Theorem \ref{thm_keyprop} by the Cauchy-Schwartz inequality since
\[
\Big |\sum_{q\in \bZ/\lambda\bZ } \gamma_{\lambda, q,\ell}\overline{\gamma_{\lambda,q+b,\ell}}e^{2\pi i q\frac {a}\lambda}	\Big |^2
 \leq \Big(\sum_{q\in \bZ/\lambda\bZ } |\gamma_{\lambda,q,\ell}|^2 \Big)^2=\|\Pi_\mu \Pr_\lambda f\|_{L^2(M)} ^4.
\]

The Laguerre bounds in \eqref{linfty} and \eqref{eq_LaguerreBound} imply for all $\epsilon>0$
	\begin{align*}
\sum_{a,b\in \bZ} \left|
		 \cL_\ell \left (\pi \frac{a^2+b^2}{|\lambda|}\right)\right|^2 &\leq  
		 1+ 
	\sum_{(a,b)\in \bZ^2\setminus \{(0,0)\}}
		 \left|
		 \cL_\ell \left (\pi \frac{a^2+b^2}{|\lambda|}\right)\right|^{1+\epsilon}\\
		 &\lesssim  
		 1+ 
\sum_{(a,b)\in \bZ^2\setminus \{(0,0)\}} \left(
 \frac{2\ell+1}{\pi \frac{a^2+b^2}{|\lambda|}}
\right)^{1+\epsilon},
		 \end{align*}	 
		 and we estimate readily
		\begin{align*}
\sum_{(a,b)\in \bZ^2\setminus \{(0,0)\}} \left(
\frac{2\ell+1}{\pi \frac{a^2+b^2}{|\lambda|}}
\right)^{1+\epsilon}
&\lesssim_\epsilon (|\lambda|(2\ell+1))^{1+\epsilon}
\sum_{(a,b)\in \bZ^2\setminus \{(0,0)\}} 
(a^2+b^2)^{-(1+\epsilon)}.
		 \end{align*}	 	 
		 Since $\sum_{(a,b)\in \bZ^2\setminus \{(0,0)\}} 
(a^2+b^2)^{-(1+\epsilon)}$ is finite, this concludes the proof. 
\end{proof}

\subsubsection{End of the proof of Theorem \ref{main}}
We want to show that 
\begin{equation}
	\label{eq_WTSmain}
\|\Pi_\mu f\|_{L^4(M)}  \lesssim_\epsilon \mu^{ \frac 1 2+\epsilon} \|f\|_{L^2(M)}.
\end{equation}
Using the projectors $\Pr_\lambda$ for the decomposition 
$L^2(M)=\oplus^\perp_\lambda L^2_\lambda(M)$
defined in \eqref{eqdef_Prlambda}, 
 we have
 \[
 \|\Pi_\mu f\|_{L^4(M)} 
=\|\Pi_\mu \sum_{\lambda\in \bZ} \Pr_\lambda f\|_{L^4(M)} 
\leq \sum_{\lambda\in \bZ}\|\Pi_\mu  \Pr_\lambda f\|_{L^4(M)}.
\]

For $\lambda=0$, since $\Pi_\mu$ and $\Pr_0$ commute, we have 
$$
\|\Pi_\mu \Pr_0 f\|_{L^4(M)}
=\|\Pr_0\Pi_\mu  f\|_{L^4(M)}
= \|\Pr_0\Pi_\mu  f|_{c=0}\|_{L^4(\bT^2)},
$$
by Lemma \ref{lem_L20identL2T2}, while by Zygmund's estimate in  \eqref{zyg2ddual}, 
$$
\|\Pr_0\Pi_\mu  f|_{c=0}\|_{L^4(\bT^2)}\leq 5^{1/4} \|\Pr_0\Pi_\mu  f_{c=0}\|_{L^2(\bT^2)}.
$$
Now, 
$$
\|\Pr_0\Pi_\mu  f_{c=0}\|_{L^2(\bT^2)} =\|\Pr_0\Pi_\mu  f\|_{L^2(M)} \leq \|\Pi_\mu  f\|_{L^2(M)},
$$
by  Lemma \ref{lem_Prlambda}. Hence, we have obtained
\begin{equation}
	\label{eq_term0}
 \|\Pi_\mu \Pr_0 f\|_{L^{4}(M)}\leq 5^{1/4}\|\Pi_\mu   \Pr_0 f\|_{L^2(M)}\leq 5^{1/4}\|f\|_{L^2(M)}.
\end{equation}
	
\medskip 

For $\lambda\neq 0$, 
 when $\Pi_\mu \Pr_\lambda \neq 0$, 
we have
by Proposition \ref{propkeyprop}, 
\[
\|\Pi_\mu \Pr_\lambda f\|_{L^4(M)}\lesssim_\epsilon \mu^{ (1+\epsilon)/2}\|\Pi_\mu \Pr_\lambda f\|_{L^2(M)},
\]
and furthermore, $|\lambda|$ divides $2\pi \mu^2$.
Consequently, 
\begin{align*}
& \sum_{\lambda\in \bZ\setminus \{0\} }
\|\Pi_\mu \Pr_\lambda f\|_{L^4(M)}
\lesssim_\epsilon \mu^{ (1+\epsilon)/2}
\sum_{\lambda\in \bZ\setminus \{0\} } \|\Pi_\mu \Pr_\lambda f\|_{L^2(M)}
\\
&\qquad \lesssim_\epsilon \mu^{ (1+\epsilon)/2}
\sqrt{\sum_{\lambda\in \bZ\setminus \{0\} } \|\Pi_\mu \Pr_\lambda f\|_{L^2(M)}^2}
\sqrt{\sum_{\substack{\lambda\in \bZ\setminus \{0\}\\
 \text{dividing}\, \frac{\mu^2 }{2\pi} }} 1},
\end{align*}
by the Cauchy-Schwartz inequality. 
By the Plancherel formula, 
\[\sum_{\lambda\neq 0 } \|\Pi_\mu \Pr_\lambda f\|_{L^2(M)}^2 
= \|\sum_{\lambda\neq 0 }  \Pi_\mu \Pr_\lambda f\|_{L^2(M)}^2 
\leq \|f\|_{L^2(M)}^2.
\]
Recall that the number-of-divisor function of a positive integer $n$ often denoted by  $\tau(n)=\sum_{k\mid n} 1$ satisfies $\tau(n)\lesssim_\delta n^\delta$, for any $\delta>0$ \cite{HW}.
Hence, we have obtained 
 \[
  \sum_{\lambda\neq 0 }
\|\Pi_\mu \Pr_\lambda f\|_{L^4(M)}
\lesssim_\epsilon \mu^{ \frac 12 +\epsilon }
\|f\|_{L^2(M)}.
\]

Combining the cases $\lambda=0$ in \eqref{eq_term0}
and the above sum over $\lambda\neq 0$ yields \eqref{eq_WTSmain} and concludes the proof of Theorem \ref{main}.

\subsection{Proof of the sharpness}

This section is devoted to the proof of Proposition \ref{prop:sharpness} that shows the sharpness of our result in Theorem \ref{main}.
It relies on noticing that the sum in $q\in \bZ/\lambda\bZ$ in \eqref{estL4} is in fact related to the discrete short-time Fourier transform.

\subsubsection{A few words about  the discrete short-time Fourier transform} 
 Recall that the discrete short-time Fourier transform  of any two functions $\gamma,\eta:\bZ/\lambda \bZ\to \bC$ is the function $V_\gamma \eta : (\bZ/\lambda \bZ)^2\to \bC$ defined via
\[
V_\gamma \eta(a,b):=\sum_{q\in \mathbb Z/\lambda\mathbb Z}
\eta (q)\,\overline{\gamma(q-a)}\,e^{-2\pi i bq/\lambda}.
\]
The finite setting considered here is a special case of Gabor analysis on
locally compact abelian groups;
references  can be found in the literature for time-frequency analysis such as
\cite{FeichKL,Gro,Gro2}.

\subsubsection{Construction of a special function}
Our proof of the sharpness (Proposition \ref{prop:sharpness}) relies on the construction of a special function $\gamma=g$ in the following statement:  

\begin{lemma}
\label{lem_Vgg}
Let
\(\lambda\in \mathbb N\), and let \(A\in \mathbb N\) with \(8A<\lambda\).
We consider the indicatrix function $g := 1_{[-A,A]}$ of the interval $[-A,A]$ as a function on $\bZ / \lambda \bZ$. 
Then we have
\[
\|g\|_{\ell^2(\mathbb Z/\lambda\mathbb Z)}^2
=
2A+1,
\]
and with $\cL_0(v)=e^{-\frac {v}2}$, 
\[
\sum_{a,b\in\mathbb Z}
\left|
V_g g(a,b)\,\mathcal L_0\!\left(\pi\frac{a^2+b^2}{\lambda}\right)
\right|^2
\ge  e^{-\pi \frac {A^2}\lambda  }  e^{-\pi \frac \lambda   {A^2}} \left(\frac{2A}{\pi}\right)^2\left(\lambda  + 2A+1\right) .
\]
\end{lemma}

\begin{remark}
The idea behind Lemma \ref{lem_Vgg} is to take advantage of  
the concentration of  the weight 
$e^{-\frac{\pi}{2}\frac{a^2+b^2}{\lambda}}$
 near the origin in \((\mathbb Z/\lambda\mathbb Z)^2\).
 Indeed, an important obstruction to making $V_g g$ small is
 the Moyal formula, which reads for any $\gamma,\eta$ 
\[
\sum_{a,b\in\mathbb Z/\lambda\mathbb Z}
|V_\gamma\eta(a,b)|^2
=
\lambda \|\gamma\|_{\ell^2}^2\|\eta\|_{\ell^2}^2 .
\]
The indicatrix
\(g=1_{[-A,A]}\) is a simple way of pushing a substantial part of the mass of $V_g g$
away from the region where the weight is large: it localises \(V_g g\) in
the translation variable \(a\in [-A,A]\), while forcing it to spread in the frequency
variable \(b\), on a scale of order \(\lambda/A\).
	\end{remark}

\begin{proof}[Proof of Lemma \ref{lem_Vgg}]
We compute readily $\|g\|_{\ell^2(\mathbb Z/\lambda\mathbb Z)}^2 = 2A+1$ and
\[
V_g g(a,b)
=
\sum_{q\in [-A,A]\cap (a+[-A,A])} e^{-2\pi i bq/\lambda},
\qquad -\frac \lambda 2 \leq a,b \leq \frac \lambda 2.
\]

Since \(8A<\lambda\), one can  check that 
\begin{itemize}
	\item if $|a|>2A$, 
	the interval $[-A,A]\cap (a+[-A,A])$ is empty, 
	so $V_g g(a,b)=0$,
	\item if \(|a|\le 2A\), 
	we may view $[-A,A]\cap (a+[-A,A])$  as a non-empty interval of $\bZ / \lambda \bZ$ which has  length $2A+1 -|a|$, so 
	\[
|V_g g(a,b)|
=
\left|
\sum_{j=0}^{2A -|a|} e^{-2\pi i bj/\lambda}
\right|
=
\left|
\frac{\sin\frac{\pi b(2A+1 -|a|)}\lambda}{\sin \frac {\pi b}\lambda}
\right|
\]
if $0<|b|\leq \frac \lambda 2$, while if $b=0$, 
\[
|V_g g(a,0)|
=
2A-|a|+1.
\]
\end{itemize} 

We consider the rectangle
\[
R_{A}:=\left \{(a,b)\in \bZ^2 \colon |a|\leq A, \ |b|\leq \frac{\lambda}{2(2A+1)}\right \}
\] 
 as a subset of $(\bZ/\lambda\bZ)^2.$
Using the well-known bounds
\[
\forall t\in [-\frac \pi 2,\frac \pi 2],\quad 
|\sin t|\ge \frac{2}{\pi}|t|
\qquad\text{and}\qquad 
\forall t\in \bR,\quad 
0\leq |\sin t|\le |t|,
\]
we obtain 
\begin{equation} \label{boundbelow} 
|V_g g(a,b)|
\ge
\left|\frac{\frac{2}{\pi}\frac{\pi b(2A+1 -|a|)}\lambda}{\frac {\pi b}\lambda}\right|
=
\frac{2}{\pi}(2A +1 -|a|)
\ge
\frac{2A}{\pi},
\end{equation} 
for  any $a,b\in 
R_{A}$
with $b\neq 0$. This estimate also holds for $b=0$ since $|V_g g(a,0)|
\geq A+1$ for $|a|\leq A$.

\smallskip 
We observe that for any $(a,b)\in R_{A}$, 
\[
\left|\cL_0\left(\pi\frac{a^2+b^2}\lambda\right)\right|^2=e^{-\pi\frac{a^2+b^2}\lambda} \geq e^{-\pi \frac {A^2}\lambda  }{e^{-\pi \frac \lambda   {A^2}}},
\]
consequently, 
\begin{align*}
\sum_{a,b\in\mathbb Z}
\left|
V_g g(a,b)\,\mathcal L_0\!\left(\pi\frac{a^2+b^2}{\lambda}\right)
\right|^2
&\ge 
\sum_{(a,b)\in R_{A}}
\left|
V_g g(a,b)\,\mathcal L_0\!\left(\pi\frac{a^2+b^2}{\lambda}\right)
\right|^2
\\
&\geq  e^{-\pi \frac {A^2}\lambda  }{e^{-\pi \frac \lambda   {A^2}}}
\sum_{(a,b)\in R_{A}}
\left|
V_g g(a,b)
\right|^2
\\
&\geq  e^{-\pi \frac {A^2}\lambda  }{e^{-\pi \frac \lambda   {A^2}}}
\left(\frac{2A}{\pi}\right)^2 |R_{A}|,
\end{align*}
by \eqref{boundbelow}, where $|R_{A}|=
(2A+1) (1+\frac \lambda {2A+1})$ denotes the cardinal of $R_{A}$.
This concludes the proof of the lemma.
\end{proof}

\subsubsection{Proof of  the  sharpness for $\mu^{1/2}$.}
Consider $\gamma=g$ given in Lemma \ref{lem_Vgg} with  $A$ equal to the integer part of $\sqrt \lambda$; hence  $A\asymp\sqrt\lambda$ for $\lambda$ large enough. 
We have 
by Theorem  \ref{thm_keyprop} 
\begin{align*}
\|\sum_{q\in \bZ/\lambda\bZ} \gamma_q h_{\lambda,q,0}
	\|_{L^4(M)}^4
	&=	
	\sum_{a,b\in \bZ} \left|
	V_\gamma \gamma (a,b)
	 \cL_0 \left (\pi \frac{a^2+b^2}\lambda\right)\right|^2\\
	 & \gtrsim A^2 \lambda \asymp \lambda^2,	
\end{align*}
	while 
	\[
	\|\sum_{q\in \bZ/\lambda\bZ} \gamma_q h_{\lambda,q,0}
	\|_{L^2(M)}^2 = \|\gamma\|_{\ell^2(\bZ/\lambda\bZ)} ^2 =2A+1 \asymp \lambda^{1/2}.
	\]
Therefore, 
\[
	\frac{\|\sum_{q\in \bZ/\lambda\bZ} \gamma_q h_{\lambda,q,0}
	\|_{L^4(M)}}
	{\|\sum_{q\in \bZ/\lambda\bZ} \gamma_q h_{\lambda,q,0}
	\|_{L^2(M)}}
	 \gtrsim 
	 \frac{\lambda ^{1/2}}{\lambda^{1/4}} =\lambda^{1/4}.
	\]		
	Here, we have $2\pi \mu^2 = \lambda$, so $\lambda^{1/4}\asymp \mu^{1/2}.$
This shows Proposition \ref{prop:sharpness}.	

\appendix

\section{Hermite and Laguerre functions}
\label{sec_hermitefcn}
\subsection{Hermite functions} 
This appendix collects some useful facts about Hermite functions. These functions arise naturally as the eigenfunctions of the harmonic oscillator and therefore play a central role in quantum mechanics. We begin by recalling that the Hermite polynomials $H_\ell(u)$ are defined on the real line by
\[
H_\ell(u)
=
(-1)^\ell e^{u^2}
\left(\frac{d}{du}\right)^\ell e^{-u^2},
\qquad \ell\in \bN_0 .
\]
They satisfy the orthogonality relations
$$
\int_{-\infty}^{+\infty} H_\ell (u) H_m(u) e^{-u^2} du
= \left\{\begin{array}{ll}
0&\mbox{if} \ \ell\neq m,\\
\sqrt \pi 2^\ell \ell! & \mbox{if} \ \ell =m,
\end{array}\right.
$$
and the recurrence relations (with the convention that $H_\ell=0$ for $\ell<0$)
$$
H_{\ell+1}(u)=2uH_\ell (u)-H'_\ell (u)\qquad\mbox{and}\quad
H'_\ell(u)=2\ell H_{\ell-1}(u).
$$
We define the associated Hermite functions
\begin{equation}  \label{hermite}  
h_\ell (u) := (\sqrt \pi 2^\ell \ell!) ^{-\frac 12} H_\ell(u) e^{-\frac {u^2}2 }, \qquad \ell\in \bN_0, \ u\in \bR,
\end{equation}  
which are Schwartz functions on $\bR$, and for which  the above recurrence relations imply (again with the conventions  
$
h_\ell=0$ for $\ell<0
$)
\begin{equation}  \label{CA}
\left(- \frac d{du} + u
\right)h_\ell = \sqrt{2 (\ell+1)} h_{\ell+1}
\quad\mbox{and}\quad
\left(\frac d{du} +  u
\right)h_\ell = \sqrt{2\ell} h_{\ell-1}.
\end{equation}
Furthermore, as $h_0 (u):= (\pi) ^{-\frac 14} e^{-\frac {u^2}2 }$ satisfy
$$ (-\frac {d^2}{du^2} + u^2) h_0 (u)=h_0 (u), $$
we get from \eqref{CA} and an obvious induction  that, for all $\ell\in \bN_0$,  \begin{equation}  \label{osc}
(-\frac {d^2}{du^2} + u^2) h_\ell = (2\ell+1) h_\ell.
\end{equation}  
The Hermite functions are then eigenfunctions of the
harmonic oscillator 
$$
H:=(-\frac {d^2}{du^2} + u^2)
$$
  which can also be written in the form
$H= \frac 1 2 \big(CA+AC\big),$
where $\ds C=- \frac d{du} + u$ and $\ds A=\frac d{du} +  u $  are the so-called   creation
and annihilation operators in quantum mechanics.  
It is also well-known that $(h_\ell)_{\ell\in \bN_0}$ is a Hilbert basis of $L^2(\bR)$  and that  the spectrum of   the harmonic oscillator
 is $\{2\ell+1 \colon \ell \in \bN_0\}$, each eigenvalue having multiplicity one;  see for instance \cite{Lewin, RS, RudinFA}.

\subsection{Laguerre functions} \label{Laguerre} 
Laguerre and Hermite functions are involved in the study of the Fourier transform on the Heisenberg group,  see for instance~\cite{EMOT, Th93}  and the references therein.

Recall that $\cL_\ell$   the Laguerre function  of type 0 and degree $\ell\in \bN_0$ is defined by 
$$
\cL_\ell(v) := L_\ell (v) e^{-\frac {v}2}, \quad 
$$
where $L_\ell$ stands for  the Laguerre polynomials of type 0 and degree $\ell$, that is 
$$ L_\ell(v):= \sum^\ell_{k=0} (-1)^{k} \left(
\begin{array}{ccc}
	\ell \\ \ell-k
\end{array}	\right)\frac {v^k} {k!}.$$
In particular 
$$ 
\cL_0(v)=e^{-\frac {v}2} .$$
\smallskip 
The  following is a well-known connection between the Hermite and Laguerre functions \cite{Folland,Th93} 
\begin{equation}  \label{rel}
\cL_\ell \left (\frac {x^2+y^2}2\right) = \int_\bR e^{ix\xi} 
h_\ell (\xi+\frac y2) 
h_\ell (\xi-\frac y2) d\xi.
\end{equation}   
This readily ensures that, for all $\ell\in \bN_0$ and $v \geq 0$, 
\begin{equation}  \label{linfty}
 |\cL_\ell (v) |   \leq 1, 
\end{equation}   
It will be also useful to recall \cite{Magnus,SZ}  that, for all $\ell\in \bN_0$ and $v \geq 0$,  
\begin{equation}  \label{prop} v \cL_\ell (v)= -(\ell +1)\cL_{\ell+1} (v)+ (2\ell+ 1) \cL _\ell (v) - \ell   \cL_{\ell-1} (v).\end{equation}
Consequently, combining \eqref{linfty} and \eqref{prop}, we obtain
\begin{equation}
\label{eq_LaguerreBound}
|\cL_\ell (v) |   \leq  C \frac{2\ell+1}{v},
\end{equation}
with a constant $C>0$ independent of $v>0$ and $\ell\in \bN_0$.

\section{Nilpotent Lie groups and their nilmanifolds}
\label{sec_HAGM}

Here we recall some basic definitions and properties of Lie groups, especially nilpotent.
We refer to textbooks such as \cite{Helgason-bk} for the general theory of Lie group and to  \cite{CorwinGreenleaf} for the particular case of nilpotent Lie groups and their nilmanifolds.

\smallskip In this paper, all the Lie groups and Lie algebras are over the field of real numbers. 

\subsection{Basic facts}
\subsubsection{Lie groups}
\label{subsecLiegr}
A \emph{Lie group} $G$ is a smooth manifold equipped with a smooth group structure.
Its \emph{Lie algebra} $\fg$ is the vector space of left-invariant vector fields equipped with the commutator bracket. It also identifies with the tangent space at the origin (neutral element) equipped with the corresponding Lie bracket. Consequently $\dim G=\dim \fg<\infty$ and 
for each vector field $X\in \fg$, we can define $\exp X$ 
as the time-1 flow from the origin. This defines  the \emph{exponential map} $\exp:\fg \to G$ which is a local diffeomorphism. At least formally, 
it is also given  by the Baker-Campbell-Hausdorff formula
in terms of the adjoint representation 
$$
\ad :\fg \longrightarrow \sL(\fg),\qquad  
X\longmapsto \ad (X)(Y)=[X,Y].
$$
It also allows us to describe the identification  between an element of the $\fg$ and a left-invariant vector field:
$$
Xf(x) = \partial_{t=0} f(x\exp tX), \quad f\in C^\infty(G), \ x\in G, \ X\in \fg.
$$

A Lie algebra or group is said to be \emph{nilpotent} 
when $\ad^N=0$ for some $N\in \bN$.
In this paper, we will always assume that a nilpotent Lie group is connected and simply connected. 
In particular, the exponential map together with  a choice of basis for $\fg$ allows us to realise a nilpotent Lie group as the manifold $\bR^n$ equipped with a group law $\bR^n\times \bR^n \to \bR^n$ that is a polynomial map. 

\smallskip
A choice of basis for $\fg$ also  leads to a corresponding Lebesgue measure on  $\fg$ and the Haar measure $dx$ on the group $G$,
hence $L^p(G)\cong L^p(\bR^n)$.
This also allows us \cite{CorwinGreenleaf}
to define the spaces 
$$
C_c^\infty(G)\cong C_c^\infty(\bR^n)
\quad \mbox{and}\quad  
\cS(G) \cong \cS(\bR^n)
$$
 of test functions which are smooth and compactly supported or Schwartz, 
and the corresponding spaces of distributions 
$$
\cD'(G)\cong \cD'(\bR^n)
\quad \mbox{and}\quad 
\cS'(G)\cong \cS'(\bR^n).
$$
Note that this identification with $\bR^n$ does not usually extend to the convolution: the group convolution, i.e. the operation between  two functions on $G$ defined formally via 
$$
 (f_1*f_2)(x):=\int_G f_1(y) f_2(y^{-1}x) dy,
$$
 is   not commutative in general whereas it is a commutative operation for functions on  the abelian group $\bR^n$.

\subsubsection{Nilmanifolds}
\label{subsec_nilmanifolds}

A nilmanifold is the quotient 
$M:=\Gamma\backslash G$
of 
a  nilpotent Lie group $G$  by a discrete subgroup $\Gamma$ of $G$.
In this paper, 
we will always assume $\Gamma$ co-compact. This means that $M$ is compact. 
A concrete example is the natural discrete subgroup of the Heisenberg group  as described in the core of the paper.  
Abstract examples and characterisations of co-compact discrete subgroup of nilpotent Lie groups may be found in    \cite{CorwinGreenleaf}.

\smallskip
An element of $M$ is a class 
$\dot x := \Gamma x$
 of an element $x$ in $G$. 
 The quotient $M$ is naturally equipped with the structure of a compact smooth manifold, as well as the action of $G$ on $M$ via
\begin{equation}
	\label{actionG}
\dot x  \longmapsto  g \cdot \dot x  = \Gamma xg^{-1} = \dot x g^{-1},  \qquad g\in G.
\end{equation}

We assume that a Haar measure on $G$ is fixed. 
Then 
$M$ inherits a measure $d\dot x$ which is invariant under the action of $G$ on $M$ \cite{Helgason-bk}.
Recall that the Haar measure~$dx$ on $G$ is unique up to a constant and, once it is fixed, $d\dot x$ is the only $G$-invariant measure on~$M$ satisfying 
for any  function $f:G\to \mathbb C$, for instance continuous with compact support,
\begin{equation}
\label{eq_dxddotx}
	\int_G f(x) dx = \int_M \sum_{\gamma\in \Gamma} f(\gamma x) \ d\dot x.
\end{equation}
We denote by $\vol (M) = \int_M 1 d\dot x$  the volume of $M$.

Recall that a nilpotent Lie group  is unimodular in the sense that its Haar measures are invariant under left and right translations. 
Hence, the measure $d\dot x$ is invariant under the  action of $G$ defined by \eqref{actionG}.

\subsection{The regular representation}
\label{subsec_RG}
The group $G$  acts unitarily on~$L^2(M)$ via the right regular representation defined  by
$$
R(g)f(\dot x) = f(\dot x g), \quad f\in L^2(M), \ g\in G, \ \dot x\in M. 
$$
It
decomposes into a countable direct sum of representations $\pi$ in $\Gh$ the dual of the group~$G$  with finite multiplicity~$m(\pi)$, see \cite{Wallach}. Moreover,
the multiplicity $m(\pi)$ may in fact be described more precisely, see \cite{Richardson}.

\begin{ex}
\label{ex_Richardson_H1}
In the case of the Schr\"odinger representation $\pi_\lambda$ of the Heisenberg group $\bH_1$, 
the method in \cite{Richardson} yields
$m(\pi_\lambda)=|\lambda|,$
for each $\lambda\in \bZ\setminus\{0\}$.
\end{ex}

Denoting by $\Gamma\backslash\Gh$ the set of these representations,
this means that $L^2(M)$ decomposes into closed $R(G)$-invariant vector subspaces:
\begin{equation}
	\label{eq_L2Mdec}
	L^2(M) = \oplus^\perp_{\pi  \in \Gamma\backslash\Gh} L^2_\pi(M), 
\end{equation}
and on each $L^2_\pi(M)$, the representation $R$ is unitarily equivalent to $m(\pi)$ copies of $\pi$:
$$
L^2_\pi(M) \sim \cH_\pi \oplus \ldots \oplus \cH_\pi =m(\pi)\cH_\pi.
$$
In the case of the standard Heisenberg nilmanifold, the decomposition in \eqref{eq_L2Mdec} is described in Section  \ref{spectral}.

Note that the regular representation induces the corresponding representations of $L^1(G)$ and of $\fg$ via
\begin{align*}
	R(\kappa) f (\dot x) &:= \int_G \kappa(y) R(y)^* f (\dot x) 
dy = 
 \int_G \kappa(y)  f (\dot x y^{-1}) 
dy, \qquad \kappa\in L^1(G),\\
R(X)&:= \partial_{t=0} R (\exp tX) , \qquad X\in \fg.
\end{align*}
The regular representation of $\fg$ will be further described in \eqref{eq_RX=XM} below. 

 \subsection{Functions and operators  on $G$ and $M$}

We say that a function $f:G\rightarrow \mathbb C$  is  $\Gamma$-left-periodic or just~$\Gamma$-periodic  when we have 
$$
\forall x\in G,\;\;\forall \gamma\in \Gamma ,\;\; f(\gamma x)=f(x).
$$
This definition extends readily to measurable functions and to distributions.  

\smallskip 
There is a natural one-to-one correspondence between the functions on $G$ which are $\Gamma$-periodic and the functions on $M$.
Indeed, for any map $F$ on $M$, 
the corresponding periodic function on $G$ is $F_G$ defined via
$$	 
F_G(x) := F(\dot x), \quad x\in G,
$$
while if $f$ is a $\Gamma$-periodic function on $G$, 
it defines a function $f_M$ on $M$ via
$$
f_M(\dot x) =f(x), \qquad x\in G.
$$
Naturally, $(F_G)_M=F$ and $(f_M)_G=f$.

We also  define, at least formally, the periodisation $\phi^\Gamma$ of a function $\phi$ by: 
$$
\phi^\Gamma(x) = \sum_{\gamma \in \Gamma } \phi(\gamma x), \qquad x\in G.
$$

If $E$ is a space of functions or of distributions on $G$, then we denote by $E^\Gamma$ the space of elements in $E$ which are  $\Gamma$-periodic. 
Although 
$C_c^\infty(G)^\Gamma = \{0\} = \cS(G)^\Gamma,$
many other periodised functions or functional spaces have interesting descriptions on $M$:
\begin{lemma}[\cite{Fischer2023}]
\label{lem_periodisationGM}
\begin{enumerate}
\item
The map $\phi \mapsto \phi^\Gamma$ yields a surjective morphism of topological vector spaces from
$\cS(G)$ onto $C^\infty(G)^\Gamma$
and from
$C_c^\infty(G)$ onto $C^\infty(G)^\Gamma$.
\item  The map $F\mapsto F_G$ yields an isomorphism of topological vector spaces
from $C^\infty(M)$ onto $C^\infty(G)^\Gamma$ and 
from $\cD'(M)$ onto $\cS'(G)^\Gamma=\cD'(G)^\Gamma$ with inverse $f\mapsto f_M$.
\item 
For every $p\in [1,\infty]$,
the map $F\mapsto F_G$ is an isomorphism of the topological vector spaces (in fact Banach spaces) from $L^p(M)$ onto $L^p_{loc}(G)^\Gamma$ with inverse $f\mapsto f_M$.
\end{enumerate}
\end{lemma}

\label{subsubsec_opGM}
A mapping $T:\cS'(G)\to \cS'(G)$
or $\cD'(G) \to \cD'(G)$ is (left-)invariant under an element $g\in G$ when  $T(f(g \, \cdot)) = (Tf)(g \, \cdot)$ for all $
f\in \cS'(G)$ (resp. $\cD'(G))$).
It is invariant under a subset of $G$ if it is invariant under every element of the subset.
For instance, (right) convolution operators $f\mapsto f*\kappa$  with $\kappa\in \cS'(G)$ are  invariant under (left-)translation under $G$.
We will be interested mainly in the following example:
\begin{ex}
Consider a linear continuous mapping $T:\cS'(G)\to \cS'(G)$
or $\cD'(G) \to \cD'(G)$ respectively
which is invariant under $\Gamma$. Then it  naturally induces
a linear continuous mapping
 $T_M:\cD'(M)\to \cD'(M)$ on $M$ given via
$$
T_M F = (TF_G)_M, \qquad F\in \cD'(M).
$$
Consequently, 
if  $T$ coincides with   a smooth differential operator 	on $G$ that is invariant under~$\Gamma$, then $T_M$ is a smooth differential operator on $M$.
\end{ex}

In particular, any left-invariant vector field $X\in \fg$ on $M$ yields a corresponding vector field $X_M$ on $G$ and we have
$$
X_M F(\dot x) = (XF_G)(x) = \partial_{t=0}F_G (x \exp tX) = \partial_{t=0} R(\exp tX )F \, (\dot x).
$$
In other words,  
\begin{equation}
	\label{eq_RX=XM}
	R(X) =X_M.
\end{equation}

\subsection{Construction of functions in  $L^2(M)$ via representations}
 
A standard method to produce functions in $L^2(M)$
 is  the following construction \cite {GGP}:
\begin{lemma}
\label{lem_phipinuh}
	Let $\pi$ be a strongly continuous unitary representation on a separable Hilbert space $\cH_\pi$ of a finite dimensional real Lie group
 $G$.
 For any smooth vector $h\in \cH_\pi^{+\infty}$ and any distributional vector $\nu\in \cH_\pi^{-\infty}$, the function defined via 
 $$
 \phi_{\pi,\nu,h}(x) := ( \nu , \pi(x) h)_{\cH_\pi^{-\infty}\times \cH_\pi^{+\infty}}, \qquad x\in G, 
 $$
 is smooth, that is,  $\phi_{\pi,\nu,h}\in C^\infty(G)$.
  Moreover, if $\nu$ is invariant under $\pi(g_0)$ for some $g_0\in G$, i.e. $\pi(g_0)\nu=\nu$, then $\phi_{\pi,\nu,h}$ is invariant under $g_0$, 
 $$
 \mbox{i.e.}\qquad \forall x\in G\qquad 
 \phi_{\pi,\nu,h} (g_0 x) = \phi_{\pi,\nu,h}(x).
 $$ 
\end{lemma}

Applying Lemma \ref{lem_phipinuh} to the case of a nilmanifold, we obtain readily:
\begin{corollary}
	\label{cor1_lem_phipinuh}
	Let $M=\Gamma\backslash G$ be a nilmanifold. 
	Let $\pi$ be a strongly continuous unitary representation of $G$ on a separable Hilbert space $\cH_\pi$.
	
	If $\nu\in \cH_\pi^{-\infty}$ is invariant under $\pi(\gamma)$, $\gamma\in \Gamma$, then 
	$(\phi_{\pi,\nu,h})_M \in C^\infty(M)$, for any $h\in \cH_\pi^{+\infty}$.
\end{corollary}

\section{Carnot groups and nilmanifolds}
\label{sec_carnot}

\subsection{Carnot groups}
\label{sec_carnotgr}

A \emph{stratified} Lie group $G$  is a connected and simply connected 
Lie group 
whose Lie algebra $\fg$ 
admits an $\bN$-gradation
$\fg= \oplus_{\ell=1}^\infty \fg_{\ell}$
where the $\fg_{\ell}$, $\ell=1,2,\ldots$, 
are vector subspaces of $\fg$
and satisfying 
$[\fg_{1},\fg_{\ell}]=\fg_{1+\ell}$
for any $\ell\in \bN$.
This implies that the group $G$ is nilpotent
and that $\fg_\ell=\{0\}$ for all $\ell>s$ for some $s\in \bN$. The smallest $s$ is called the step of nilpotency, and $\fg_\ell\neq \{0\}$ for all $\ell \leq s$.
When a basis $\{X_1,\ldots X_{n_1}\}$ of $\fg_1$ is fixed, the group is said to be \emph{Carnot}. 
They are the model of sub-Riemannian manifolds \cite{ABB19}.
Examples of such groups are the Heisenberg group.

The structure of Carnot groups allows us to define the associated sub-Laplacian  
\begin{equation}
	\label{eqdef_sublap}
\cL = -(X_1^2+\ldots + X_{n_1}^2),
\end{equation}
as well as the homogeneous structure defined below. 
The sub-Laplacian $\cL$ is a smooth differential operator on $M$ that  is essentially self-adjoint on $\cS(G)\subset L^2(G)$ \cite{Folland2}.
We keep the same notation for their self-adjoint extension.

Let us now discuss the homogeneous structure of Carnot groups. 
For any $r>0$, 
we define the  linear mapping $D_r:\fg\to \fg$ by
$D_r X=r^\ell X$ for every $X\in \fg_\ell$, $\ell\in \bN$.
Then  the Lie algebra~$\fg$ is endowed 
with the family of dilations~$\{D_r, r>0\}$
and becomes a homogeneous Lie algebra in the sense of 
\cite{FollandStein}.
The associated group dilations 
still denoted by $D_r$ and obtained via the exponential mapping 
leads in a canonical way to the notions of homogeneity for functions, distributions and operators. 
As  examples, 
the sub-Laplacian $\cL$ is homogeneous of degree 2 while 
the Haar measure is $Q$-homogeneous, where
$$
Q:=\sum_{\ell\in \bN}\ell \dim \fg_\ell,
$$
 is called the homogeneous dimension of $G$.

\subsection{Sub-elliptic Bernstein inequalities}
\label{subsec_Bernstein}

A  nilmanifold $M=\Gamma \backslash G $ where $G$ is Carnot is also called Carnot.
The sub-Laplacian $\cL$ on $G$ defined in \eqref{eqdef_sublap} is invariant under left-translation. As seen in Section \ref{subsubsec_opGM}, it induces 
the sub-Laplacian  on $M$
$$
\cL_M = -(X_{1,M}^2+\ldots + X_{n_1,M}^2).
$$
This is a smooth differential operator which is essentially self-adjoint on $C^\infty(M)\subset L^2(M)$~\cite{Fischer2023};
we will keep the same notation for $\cL_M$ and its self-adjoint extension.
The spectrum  of 	$\cL_M$ is a discrete and unbounded subset of $[0,+\infty)$, 
each eigenspace of $\cL_M$ has finite dimension and 
the constant functions on $M$ form the 0-eigenspace of $\cL_M$ \cite{Fischer2023}.
A consequence of the semiclassical analysis in \cite{Fischer2023} (see also \cite{FM}) is the following statement:
\begin{proposition}[\cite{Fischer2023}]
\label{prop_nilmanifold}
 Let  $\psi\in \cS(\bR)$. 
\begin{enumerate}
\item The integral kernel of $\psi (\cL_M)$ is the smooth function $K_\psi\in C^\infty (M\times M)$ 
given by 
$$
K_\psi(\dot g,\dot h) = \sum_{\gamma\in \Gamma} \kappa_\psi(g^{-1}\gamma h), \quad g,h\in G.
$$
\item Set $\psi_\eps  := \psi (\eps^2 \cdot)$.
We have:
$$
K_{\psi_\eps}(\dot g,\dot g)
= \eps^{-Q}\kappa_\psi (0) +O(\eps)^\infty,
$$
in the sense that 
$$
\forall N\in \bN,  \quad \exists C>0, \quad \forall \dot g\in M, \, \forall \eps\in (0,1],\quad 
|K_{\psi_\eps}(\dot g,\dot g)
- \eps^{-Q}\kappa_\psi (0) |\leq C \eps^N.
$$
\end{enumerate}
\end{proposition}

Proposition \ref{prop_nilmanifold} readily implies:
\begin{corollary}
\label{cor_prop_nilmanifold}
For any  $p\in [2,\infty]$and any $\psi\in \cS(\bR)$, there exists $C=C_{p,\psi}>0$ such that the following estimate holds  for any $\eps\in (0,1]$, 
$$
\|\psi (\eps^2\cL_M) \|_{\sL(L^2(M),L^p(M))}
\leq C \eps ^{Q(\frac 1p -\frac 12)}.
$$
\end{corollary}

Before giving the proof of Corollary \ref{cor_prop_nilmanifold}, let us explain why its statement can be interpreted as the analogous of Bernstein inequality in the Euclidean framework (for a brief review on this topic, one can consult \cite{bahouri lp} and the references therein).

\smallskip
We  fix an orthonormal basis $e_\ell$, $\ell\in \bN_0$,  of $L^2(M)$ such that:
$$
\cL_M e_\ell = \lambda_\ell^2 e_\ell.
$$
Then any function $f\in L^2(M)$ may be written as a Hilbertian series 
\begin{equation}
	\label{eq_Hilbertseries}
	f = \sum_{\ell=0}^\infty (f,e_\ell)_{L^2(M)} e_\ell 
\quad\mbox{in}\ L^2(M).
\end{equation}
We say that $f$ is spectrally supported in $I\subset \bR$ when the sum above is over $\ell$ such that~$\lambda_\ell \in I$;
this is independent of the chosen basis $(e_\ell)$.
 In particular, a function $f\in L^2(M)$ is said to be spectrally supported in~$[0,\Lambda]$,~$\Lambda>0$, if and only if, for all~$\phi \in L^\infty(\bR)$ identically equal to $1$ in~$ [0,1]$, we have \begin{equation}
\label{defspesup} f= \phi(\Lambda^{-2} \cL_M)f. \end{equation} 
In this case, applying Corollary \ref{cor_prop_nilmanifold} with~$\eps = \Lambda^{- 1}$ and $\phi=\psi\in C_c^\infty(\bR)$ such that $\psi=1$ on~$[0,1]$ and $\psi=0$ outside $[-1,2]$, 
we have $\psi(\Lambda^{-2} \cL_M) f = f$ so 
$$
\|f\|_{L^p(M)}  \leq C \Lambda^{Q (\frac 12 -\frac 1p)}
\|f\|_{L^2(M)} .
$$

\begin{proof}[Proof of Corollary \ref{cor_prop_nilmanifold}]
For any bounded function $\psi:\bR\to\bR$, 
the operator $\psi(\cL_M)$ applied to~$f\in L^2(M)$ decomposed in the  Hilbertian series in \eqref{eq_Hilbertseries} as$$
\psi(\cL_M)f = \sum_{\ell=0}^\infty \psi(\lambda_\ell^2) (f,e_\ell)_{L^2(M)} e_\ell .
$$ 
We deduce that $\psi(\cL_M)$ is bounded on $L^2(M)$ with
\begin{equation}
	\|\psi (\cL_M)\|_{\sL(L^2(M))}
\leq \sup_{\lambda\in \bR} |\psi (\lambda^2)|,	
\label{eq_FCDelta}
\end{equation}
and that
 its integral kernel $K_\psi$, that is, the distribution defined via
\begin{equation}
\psi(\cL_M)f (\dot x) = \int_M K_\psi(\dot x,\dot y) f(\dot y) d\dot y,\quad f\in C^\infty (M),
\label{eq_kernelM}
\end{equation}
by Proposition \ref{prop_nilmanifold} (1),
 is formally given by 
\begin{equation}
	\label{eq_Kpsi}
	K_\psi (\dot x,\dot y) = \sum_{\ell=0}^\infty  \psi(\lambda_\ell^2)
\overline{e_\ell (\dot x)} e_\ell(\dot y).
\end{equation}
Applying  the Cauchy-Schwartz inequality in \eqref{eq_kernelM} implies readily 
$$
\|\psi_\eps  (\cL_M) \|_{\sL(L^2(M),L^\infty(M))} 
\leq   \sup_{\dot x\in M} \|K_{\psi_\eps }(\dot x, \cdot)\|_{L^2(M)}.
$$
Now by \eqref{eq_Kpsi}, we see
\begin{align*}
\|K_{\psi_\eps }(\dot x, \cdot)\|_{L^2(M)}^2
&=
\int_M 
\sum_{\ell,\ell'}  \psi_\eps (\lambda_\ell^2)\overline{\psi_\eps (\lambda_{\ell'}^2)}\
\overline{e_\ell (\dot y)} e_\ell(\dot x)\
e_{\ell'} (\dot y) \overline{e_{\ell'}(\dot x)} d\dot y 	
\\
&=\sum_{\ell}|\psi_\eps| ^2(\lambda_\ell^2)
e_\ell(\dot x)
 \overline{e_\ell(\dot x)},
\end{align*}
since 
$\int_M \overline{e_\ell (\dot y)} e_{\ell'} (\dot y) d\dot y 	
=\delta_{\ell,\ell'}$ by the Plancherel formula.
We recognise 
$$
\|K_{\psi_\eps}(\dot x, \cdot)\|_{L^2(M)}^2 = |K_{|\psi_\eps|^2}(\dot x, \dot x)|.
$$
Hence, 
$$
\|\psi_\eps  (\cL_M) \|_{\sL(L^2(M),L^\infty(M))}
\leq   \sqrt{\sup_{\dot x \in M} 
|K_{|\psi_\eps|^2}(\dot x, \dot x)|}.
$$ 
We deduce that 
 \begin{align*}
 \|\psi_\eps  (\cL_M) \|_{\sL(L^2(M),L^\infty(M))}
& \leq C \eps^{-\frac Q2},	
 \end{align*}
by Proposition \ref{prop_nilmanifold} (2)
for some positive constant $C$.
This shows the case $p=+\infty$.
The case~$p=2$ follows from \eqref{eq_FCDelta} applied to $\psi_\eps$, with $C = \sup_\bR |\psi| =\sup_\bR |\psi_\eps|  $.
By interpolation, we obtain the case of any $p\in (2,\infty)$.
\end{proof}

The Bernstein inequality now follows readily:

\begin{theorem}[Sub-elliptic Bernstein Inequality]
\label{thm_Bernstein}
Let $M=\Gamma \backslash G$ be a compact nilmanifold on a Carnot group $G$. 	
For any $p\in [2,\infty]$
there exists a  constant $C>0$ such that the following estimate holds for any $\mu>1$:
$$
\|\Pi_{\mu,\sqrt{\cL_M}}\|_{\sL(L^p(M), L^2(M))} \leq C  \mu^{Q(\frac 12 -\frac 1p)}.
$$	
\end{theorem}
\begin{proof}
	We have
	$$
	\Pi_{\mu,\sqrt{\cL_M}} f = \sum_{\ell : \lambda_\ell =\mu}  (f,e_\ell)_{L^2(M)} e_\ell,
	\qquad\mbox{so}\qquad 
	\|\Pi_{\mu,\sqrt{\cL_M}} f\|_{L^2(M)}=^2\sum_{\ell : \lambda_\ell =\mu}  |(f,e_\ell)_{L^2(M)}|^2,
	$$
	by the Plancherel formula,
	and therefore
	$$
	\|\Pi_{\mu,\sqrt{\cL_M}} f\|_{L^2(M)}^2
	\leq \sum_{\ell =0}^{\infty} |\psi|^2(\mu^{-2}\lambda_\ell^2) |(f,e_\ell)_{L^2(M)}|^2 = \|\psi(\mu^{-2} \cL) f\|_{L^2(M)}^2,
	$$
	for a function $\psi\in C_c^\infty(\bR)$ chosen such that $\psi=1$ on $[0,1]$ and $\psi=0$ outside $[-1,2]$. We conclude with applying Corollary \ref{cor_prop_nilmanifold} to $\eps = \mu^{-1}$.
\end{proof}
\begin{remark}
\label{rem_pfproptrivial}
	Applying Theorem \ref{thm_Bernstein} to the standard nilmanifold of the three dimensional Heisenberg group with $p=4$ and $Q=4$ shows Proposition \ref{proptrivial}. 
\end{remark}

\section{End of the proof of Theorem \ref{thm_decR}}	
\label{sec_BWZ}

Here, we continue the proof of Theorem \ref{thm_decR}. This leads naturally to the definition of our BWZ transform. 
At this point in the proof, 
we look for an explicit decomposition of $L^2_\lambda(M)$ when $\lambda\neq0$.
 By harmonic analysis, we already know that such an abstract decomposition exists: indeed, 
by Lemma~\ref{lem_RactsL2lambda} and  the Stone-Von Neumann theorem~\cite{Folland}, when $\lambda\neq 0$, $R$ decomposes as copies of the Schr\"odinger representation~$\pi_{\lambda}$ of $\bH_1$, that is, the 
  unique (up to unitary equivalence)  unitary irreducible representation of $\bH_1$ satisfying on the centre~$ \pi_{\lambda}(0,0,c)  =e^{2\pi i \lambda c}\id$. 
  Here, 
  we will obtain this decomposition concretely. 
Our strategy consists in applying the construction explained in Corollary \ref{cor1_lem_phipinuh} to the Schr\"odinger representations  
 $\pi=\pi_\lambda$ introduced in Section  \ref{subsec_schro}. 
 This requires first to determine its distributional vectors  invariant under $\pi_\lambda(\gamma)$, $\gamma\in \Gamma$, which we now do.

\subsection{Further properties of $\pi_\lambda$}

The Schr\"o\-dinger representation $\pi_\lambda$ acts unitarily on $L^2(\bR)$. 
Its space of smooth vectors is the Schwartz space  $\cS(\bR)$, 
 while its space of distributional vectors is the space  $\cS'(\bR)$ of tempered distributions. 
The action of the Schr\"odinger representations is continuous on the Fr\'echet space $\cS(\bR)$ and extends continuously to $\cS'(\bR)$. We say that a distribution $\nu\in \cS'(\bR)$ is invariant under $\pi_\lambda(g)$ for some~$g\in \bH_1$ when $\pi_\lambda(g) \nu = \nu$, meaning that, for all $h\in \cS(\bR)$,
$$
(\pi_\lambda(g) \nu ,h)_{\cS'(\bR)\times \cS(\bR)}
=(\nu ,h)_{\cS'(\bR)\times \cS(\bR)}\, .
$$ 
The description of the distributions invariant under 
$\pi_\lambda(\gamma)$ for all $\gamma\in \Gamma$ is well-known with different conventions, and we  also provide a proof below: 

\begin{proposition} [\cite{KTX, Th2009}]
\label{properties}
For each $\lambda\in \bZ\setminus\{0\}$ and $q\in \bZ$
the tempered distribution on $\bR$ 
\begin{equation}   \label{distimpt}
\nu_{\lambda,q} := 
\lambda^{-1} 
\sum_{k\in \bZ} 
e^{-2\pi i q \frac {k} \lambda  }\delta_{\frac {k}\lambda},
\end{equation} 
 or alternatively defined via
$$
(\nu_{\lambda,q},h)_{\cS'(\bR)\times \cS(\bR)}
=\sum_{k\in \bZ}  \widehat h(q+\lambda k), \qquad h\in \cS(\bR), 
$$ 
is invariant under $\pi_\lambda(\gamma)$, $\gamma\in \Gamma$. 
These distributions are constant over $\bZ / \lambda\bZ$-classes in $q$ in the sense that
$$
\forall \lambda\in \bZ\setminus\{0\},\ q\in \bZ, \ k_1\in \bZ,\qquad 
\nu_{\lambda, k_1 \lambda+q} = \nu_{\lambda,q}.
$$
They also satisfy for any $\lambda\in \bZ\setminus\{0\}$, $q,q_1\in \bZ$,
$$
\pi_\lambda\left (\frac {q_1}\lambda, 0,0\right) \nu_{\lambda,q} =e^{2\pi i q \frac {q_1} \lambda  }\nu_{\lambda,q}
\qquad\mbox{and}\qquad 
\pi_\lambda\left (0,\frac {q_1}\lambda, 0\right) \nu_{\lambda,q} (u)
=
e^{2\pi i \lambda  u\frac {q_1}\lambda }
\nu_{\lambda,q}(u).
$$
Moreover, they span the space $\cS'(\bR)^{\pi_\lambda(\Gamma)}$ of tempered distributions on $\bR$ invariant under $\pi_\lambda(\gamma)$, for~$\gamma\in \Gamma$, and we have
$$
\dim \cS'(\bR)^{\pi_\lambda(\Gamma)} = |\lambda|.
$$
\end{proposition}

\begin{proof}[Proof of Proposition \ref{properties}]
We check readily that each $\nu_{\lambda,q} $ is a 
tempered distribution invariant under $\pi_\lambda(\gamma)$, for~$\gamma\in \Gamma$, 
that it satisfies $\nu_{\lambda, k_1 \lambda+q} = \nu_{\lambda,q}$, 
and
\begin{align*}
\pi_\lambda\left (\frac {q_1}\lambda, 0,0\right) \nu_{\lambda,q}(u) 
&=
\lambda^{-1} 
\sum_{k\in \bZ} 
e^{-2\pi i q \frac {k} \lambda  }\delta_{\frac {k}\lambda}\left (u+\frac {q_1}\lambda\right )
=
\lambda^{-1} 
\sum_{k\in \bZ} 
e^{-2\pi i q \frac {k} \lambda  }\delta_{\frac {k +q_1}\lambda}\left (u\right )
\\
&=
\lambda^{-1} 
\sum_{k_1\in \bZ} 
e^{-2\pi i q \frac {k_1- q_1} \lambda  }\delta_{\frac {k_1}\lambda}\left (u\right )
=e^{2\pi i q \frac {q_1} \lambda  }\nu_{\lambda,q}(u).
\end{align*}
We also have
\begin{align*}
\pi_\lambda\left (0,\frac {q_1}\lambda, 0\right) \nu_{\lambda,q} (u)
&=\lambda^{-1} 
\sum_{k\in \bZ} 
e^{-2\pi i q \frac {k} \lambda  }\pi_\lambda\left (0,\frac {q_1}\lambda, 0\right)
\delta_{\frac {k}\lambda}
\\&
=\lambda^{-1} 
\sum_{k\in \bZ} 
e^{-2\pi i q \frac {k} \lambda  }e^{2\pi i \lambda  u\frac {q_1}\lambda }\delta_{\frac {k}\lambda}
=e^{2\pi i \lambda  u\frac {q_1}\lambda }
\nu_{\lambda,q}(u).
\end{align*}

We now prove that these distributions span 
$\cS'(\bR)^{\pi_\lambda(\Gamma)}$.

\smallskip Let $\nu\in \cS'(\bR)$ be invariant under $\pi_\lambda(\gamma)$, $\gamma\in \Gamma$. 
As $\nu$ is invariant under $\pi_\lambda(a,0,0)$, $a\in \bZ$, 
it is $\bZ$-periodic distribution. We write its  Fourier series expansion as
$$
(\nu,h)_{\cS'(\bR)\times \cS(\bR)} = \sum_{k\in \bZ} c_k \widehat h(k), \qquad h\in \cS(\bR),
$$
for some unique coefficients $c_k\in \bC$ depending only on $\nu$.
Above, $\widehat h=\cF_\bR h$ is the Euclidean Fourier transform of $h$:
$$
\cF_\bR h(\xi)=
\widehat h (\xi) =\int_\bR e^{-2\pi i \xi u } h(u) du.
$$
As $\nu$ is invariant under $\pi_\lambda(0,b,0)$, $b\in \bZ$,
it  satisfies 
$$
\nu(u) =e^{i 2\pi \lambda ub} \nu(u) 
\quad \mbox{in}\ \cS'(\bR), \ \mbox{for any} \ b\in \bZ.
$$
The left-hand side means 
$$
(\nu,h)_{\cS'(\bR)\times \cS(\bR)} = \sum_{k\in \bZ} c_k \widehat h(k), \quad \quad\mbox{for any} \ h\in \cS(\bR), \ b\in \bZ,
$$
while the right-hand side means for any $h\in \cS(\bR)$ and any $b\in \bZ$
\begin{align*}
( \nu,e^{i 2\pi \lambda b\cdot } h)_{\cS'(\bR)\times \cS(\bR)}
& =\sum_{k\in \bZ} c_k \cF_\bR(e^{i 2\pi \lambda b\cdot } h)(k)=\sum_{k\in \bZ} c_k \widehat h(k -\lambda b ) = \sum_{k\in \bZ} c_{k+\lambda b} \widehat h(k ).
\end{align*}
Consequently, 
$c_{k  +\lambda b} = c_{k}$ for any $b,k\in \bZ$ so
\begin{align*}
	(\nu,h)_{\cS'(\bR)\times \cS(\bR)} 
	&=\sum_{0\leq q<|\lambda|} \sum_{k_1\in \bZ}  c_{q+\lambda k_1} \widehat h(q+\lambda k_1)
	=
	  \sum_{0\leq q<|\lambda|} c_q \sum_{k_1\in \bZ}  \widehat h(q+\lambda k_1) 
	  \\&= \sum_{0\leq q<|\lambda|} c_q (\nu_{\lambda,q},h)_{\cS'(\bR)\times \cS(\bR)},
\end{align*}
where $\nu_{\lambda,q}\in \cS'(\bR)$ is defined via:
\begin{align*}
(\nu_{\lambda,q},h)_{\cS'(\bR)\times \cS(\bR)}
&=\sum_{k_1\in \bZ}  \widehat h(q+\lambda k_1)	
=\sum_{k_1\in \bZ}   \cF_\bR\left (\frac{e^{-2\pi i q \frac  \cdot \lambda}}\lambda  h (\frac \cdot \lambda) \right )(k_1)	
\\
&=  \sum_{k_2\in \bZ} 
\frac {e^{-2\pi i q \frac {k_2} \lambda  }}
{\lambda} h\left (\frac {k_2}\lambda \right )
\end{align*}
by the Poisson summation formula.
We have obtained
$\nu = \sum_{0\leq q<|\lambda|} c_q \nu_{\lambda,q}$. 
This concludes the proof of the proposition. 
\end{proof}

\subsection{The Berezin-Weil-Zak transform}

For each $\lambda\in \bZ\setminus\{0\}$, 
$\nu \in \cS'(\bR) ^{\pi_\lambda(\Gamma)}$
	and $h\in \cS(\bR)$,
we consider the functions constructed in Corollary \ref{cor1_lem_phipinuh}
$$
(\phi_{\pi_\lambda, \nu,h})_M=
f_{\lambda, \nu,h}\in C^\infty(M),
\qquad \lambda\in \bZ\setminus\{0\}, \  
\nu \in \cS'(\bR) ^{\pi_\lambda(\Gamma)},\
h\in \cS(\bR),
$$
namely, for all $\dot x \in M$, 
$$f_{\lambda, \nu,h}(\dot x)=(\nu , \pi_\lambda(x) h)_{\cS'(\bR)\times \cS(\bR)}.$$

\begin{lemma}
	\label{lem_appcor_lem_phipinuh}
	Let $\lambda\in \bZ\setminus\{0\}$, $\nu \in \cS'(\bR) ^{\pi_\lambda(\Gamma)}$ and $
h\in \cS(\bR)$.
Then $f_{\lambda, \nu,h}$ is in $C^\infty(M)\cap L^2_\lambda(M)$ and satisfies the following properties:
	\begin{enumerate}
	\item We have, for any $g_0\in \bH_1$, 
$$
R ({g_0}) f_{\lambda,\nu,h} = f_{\lambda,\nu,\pi_\lambda (g_0) h},
\quad\mbox{i.e.}\quad  \forall \dot x \in M,\quad 
f_{\lambda,\nu,h} (\dot x g_0) =f_{\lambda,\nu,\pi_\lambda (g_0) h} (\dot x).
$$	
\item In the case where $g_0=\left (a_0,0,0\right )$ with $a_0\in \bR$, we have 
$$
 f_{\lambda,\nu,h} \left(\dot x \left (a_0,0,0\right )\right )
 = e^{-2\pi i\lambda a_0 b} f_{\lambda,\pi_\lambda(a_0,0,0) \nu,h}(\dot x), \qquad \Gamma (a,b,c)=\dot x\in M.
 $$
\end{enumerate}
\end{lemma}

\begin{proof}[Proof of Lemma \ref{lem_appcor_lem_phipinuh}]
By construction $f_{\lambda, \nu,h}\in C^\infty(M)$ and 
we can easily  check that
$$
 \phi_{\pi_\lambda,\nu,h}(a,b,c) = e^{2\pi i \lambda c}( \nu , \pi_\lambda(a,b,0) h)_{\cS'(\bR)\times \cS(\bR)}= e^{2\pi i \lambda c} \phi_{\pi_\lambda,\nu,h}(a,b,0),
$$	
which in view of \eqref{defsub} shows that  $f_{\lambda, \nu,h}\in  L^2_\lambda(M)$.
The properties in    (1) follow readily by construction. 
For (2),  we observe with $x=(a,b,c)\in \bH_1$
$$
	(a_0,0,0)^{-1} x (a_0,0,0) = (a,b,c) (0,0,-a_0 b),
$$
so
\begin{align*}
	f_{\lambda,\nu,h} \left(\dot x \left (a_0,0,0\right )\right )
	&=
	e^{-2\pi i\lambda a_0 b}
	(\nu, \pi_\lambda(a_0,0,0)\pi_\lambda(x) h)_{\cS'(\bR)\times \cS(\bR)}\\
	&=e^{-2\pi i\lambda a_0 b}
	(\pi_\lambda(a_0,0,0) \nu, \pi_\lambda(x) h)_{\cS'(\bR)\times \cS(\bR)} = e^{-2\pi i\lambda a_0 b} \ds f_{\lambda, \pi_\lambda(a_0,0,0) \nu,h}(\dot x).
\end{align*}
\end{proof}

We define the generalised Berezin-Weil-Zak transform, in short \BWZ, of a function $h\in \cS(\bR)$ with parameter $\lambda\in \bZ\setminus\{0\}$ and $q\in \bZ$ as  
$\BWZ_{\lambda,q}(h) :=   c_{\lambda,q} f_{\lambda, \nu_{\lambda,q},h},$
that is, 
\begin{equation}
\label{eq:BWZapp}
\BWZ_{\lambda,q}(h)(\dot x) =
c_{\lambda,q} f_{\lambda, \nu_{\lambda,q},h} (\dot x)= 
c_{\lambda,q}
(\nu_{\lambda,q}, \pi_\lambda(x) h)_{\cS'(\bR)\times \cS(\bR)}	, \quad x=(a,b,c)\in \bH_1.
\end{equation}
Above $c_{\lambda,q} $ is the normalising constant:
$c_{\lambda,q} := \sqrt{|\lambda|}.$
Expanding \eqref{eq:BWZapp} using Proposition \ref{properties}, 
we obtain the expression \eqref{eq:BWZ} given in Theorem~\ref{thm_decR}.

\subsection{End of the proof of Theorem \ref{thm_decR}}

We have already proved Part (1). 
Part (2)  (a) and (c) follow from Lemma  \ref{lem_appcor_lem_phipinuh} and  Proposition \ref{properties}.
In order to prove Part (2) (b), it suffices to show  \eqref{orthBVZ}. 
As $\BWZ_{\lambda_i,q_i}(h_i)\in L^2_{\lambda_i}(M)$, $i=1,2$, and $L^2_{\lambda_1}(M)\perp L^2_{\lambda_2}(M)	$ for~$\lambda_1\neq \lambda_2$, we have
$$
(\BWZ_{\lambda_1,q_1}(h_1)	, \BWZ_{\lambda_2,q_2}(h_2))_{L^2(M)}=0,
$$
when $\lambda_1\neq \lambda_2$.
This can also be checked more directly with 
$$
I:=(\BWZ_{\lambda_1,q_1}(h_1)	, \BWZ_{\lambda_2,q_2}(h_2))_{L^2(M)}
=\int_{c=0}^1 
e^{2\pi i (\lambda_1-\lambda_2) c} \ldots dc,
$$
which will be zero when $\lambda_1\neq \lambda_2$. 
Hence, we may assume   $\lambda_1= \lambda_2= \lambda$. 
Making the change of variable $\dot x = \dot x_1 \left (\frac {q} \lambda,0,0\right)$, we obtain by Part (2) (c) for any $q\in \bZ$:
\begin{align*}
	I&=
	\int_M \BWZ_{\lambda,  q_1}(  h_1 ) \left (\dot x_1 \left (\frac {q} \lambda,0,0\right)\right) \overline {\BWZ_{\lambda,q_2}(h_2)\left (\dot x_1 \left (\frac {q} \lambda,0,0\right)\right)}d\dot x_1\\
	&= e^{2\pi i (q_1-q_2) \frac q \lambda}I.
\end{align*}
Hence, $I=0$ when $q_1\neq q_2$, and  we may assume that $q_1=q_2=q$. 
We are left with computing 
$$
I=|c_{\lambda,q}|^2 \lambda^{-2} \sum_{k_1,k_2\in \bZ} 
 I_{k_1,k_2},
$$
where
$$
	I_{k_1,k_2}:= 
	 e^{- i 2\pi \frac {q} \lambda (k_1-k_2)} \iint_{a,b=0}^1
 e^{i 2\pi b(k_1-k_2)}
 h_1 \Big(\frac {k_1}\lambda + a\Big)
\bar h_2  \Big(\frac {k_2} \lambda +a\Big)dadb.
$$
By Fourier analysis, the integral over $b$ is zero for $k_1\neq k_2$. 
For $k=k_1=k_2$, we are left with
\begin{align*}
		I_{k,k}&=\int_{a=0}^1
 (h_1 \bar h_2)\left (\frac {k}\lambda + a\right ) da= \int_{\frac {k}\lambda}^{\frac {k}\lambda+1}
 (h_1 \bar h_2)(a' )da',
\end{align*}
after the change of variable $a'=\frac {k}\lambda + a$.
Hence, we obtain 
$$
\sum_{k\in \bZ} 
 I_{k,k}
 = \sum_{k\in \bZ} \int_{\frac {k}\lambda}^{\frac {k}\lambda+1}
 (h_1 \bar h_2)(a' )da'
 =    \sum_{k_1\in \bZ} \sum_{0\leq r<|\lambda|}
 \int_{\frac {r+ |\lambda| k_1}\lambda}^{\frac {r+ |\lambda| k_1}\lambda+1}
 (h_1 \bar h_2)(a' )da',
 $$
thanks to the Euclidean division $k=|\lambda| k_1+r$ by $\lambda$. 
Swapping the sums over $k_1$ and $r$, we obtain
$$
\sum_{k\in \bZ} 
 I_{k,k} = |\lambda| \int_{-\infty}^{+\infty}(h_1 \bar h_2)(a' )da'.
 $$
 This proves~\eqref{orthBVZ} and explains the choice of the normalising
constant $c_{\lambda,q}$, since
$
|c_{\lambda,q}|^2\lambda^{-2}|\lambda|=1.$
This proves Part~(2).

We have therefore shown the inclusion
$
\oplus_{q\in \mathbb Z/\lambda\mathbb Z}^{\perp} L^2_{\lambda,q}
\subseteq L^2_\lambda(M).$
To see that the equality holds, it is enough to use the fact that
$L^2_\lambda(M)$ contains exactly $|\lambda|$ copies of $\pi_\lambda$, 
see~Section \ref{subsec_RG} in particular Example~\ref{ex_Richardson_H1}.
This concludes the proof of 
Theorem \ref{thm_decR}.

 \end{document}